\documentclass[final,3p,onecolumn,times]{elsarticle}
\usepackage[colorinlistoftodos,textwidth=4cm,shadow]{todonotes}
\usepackage{amsmath}
\usepackage[noend]{algorithmic}
\usepackage{amsfonts}
\usepackage{color}
\usepackage{lineno,hyperref}
\usepackage{lipsum}
\usepackage{amsthm}
\usepackage{lineno,hyperref}
\usepackage{float}

\journal{Journal of Computational Physics}

\begin{document}
\begin{frontmatter}

\title{A low-rank approach to the computation of path integrals}
\author[skoltech]{Mikhail S. Litsarev}
\ead{m.litsarev@skoltech.ru}
\author[skoltech,inm]{Ivan V. Oseledets}
\address[skoltech]{Skolkovo Institute of Science and Technology, 
Skolkovo Innovation Center, Building 3, 143026 Moscow, Russia}
\address[inm]{Institute of Numerical Mathematics of Russian Academy of Sciences,
Gubkina St. 8, 119333 Moscow, Russia}

\begin{abstract}
We present a method for solving the reaction-diffusion 
equation with general potential  in free space.
It is based on the approximation of the Feynman-Kac formula
 by a sequence of convolutions
on sequentially diminishing grids.
For computation of the convolutions we propose a fast
algorithm based on the low-rank approximation of the Hankel matrices.
The algorithm has complexity of
$\mathcal{O}(nr M \log M + nr^2 M)$ flops 
and requires $\mathcal{O}(M r)$ floating-point numbers in memory,
where $n$~is the dimension of the integral, $r \ll n$, and $M$ is the mesh size in one dimension.
The presented technique can be generalized to the higher-order
diffusion processes.
\end{abstract}


\end{frontmatter}

\today

\section{Introduction}
\label{IntroSect}

Path integrals~\cite{feynm-path-1948, fh-quantmech-1965, gc-hampath-1966} 
play a dominant role in description of a wide range of problems in physics and mathematics.
They are a universal and powerful tool 
for condensed matter and  high-energy physics,
theory of stochastic processes 
and parabolic differential equations,
financial mathematics, quantum chemistry and many others. 
Different theoretical and numerical approaches 
have been developed for their 
computation, such as
the perturbation theory~\cite{agd-greenbook-1975},
the stationary phase approximation~\cite{mahan-manypart-2000, bh-asimpthotics-2010},
the functional renormalization group~\cite{jzj-qftcf-1996, jp-frgm-2003}, 
various Monte Carlo~\cite{mc-proceedings-2012} and
sparse grids methods~\cite{mh-sparse-2011, sp-proceedings-2013}.
The interested reader can find particular details in 
the original reviews and books~\cite{wong-pathrev-2014, mm-pathbook-2008, hk-pathint-2009}.

In this paper we focus on the one-dimensional reaction-diffusion equation
with initial distribution $f(x): \mathbb{R} \to \mathbb{R}^{+}$
and a constant diffusion coefficient~$\sigma$
\begin{equation}
\label{UxtDiffEqFull}
\left\{
\begin{aligned}
 & \frac{\partial}{\partial t} u(x,t)
 = \sigma \frac{\partial^2}{\partial x^2} u(x,t) - V(x,t) u(x,t),\\
& u(x,0)=f(x)
\end{aligned}
\right.
\qquad t \in [0, T],
\quad x \in \mathbb{R}.
\end{equation}
This equation may be treated in terms of a Brownian particle 
motion~\cite{crank-diffus-1975, bass-diffus-1998, chd-pathint-2001},
where the solution $u(x,t): \mathbb{R} \times [0,T]  \to \mathbb{R}^{+}$ is
the density distribution of the particles.
The potential (or the dissipation rate) $V(x,t)$ is bounded from below.
We do not consider the drift term $\rho \frac{ \partial  }{ \partial x }u(x,t)$
because it can be easily excluded by 
a substitution $u(x,t)\to \tilde u(x,t) e^{-\rho x}$~\cite{bs-brownmot-2002}.

The solution of~\eqref{UxtDiffEqFull} can be expressed by the Feynman-Kac 
formula~\cite{bs-brownmot-2002, kac-ipath-1951, ks-browmot-1991} 
\begin{equation}
\label{FKformulaF}
u(x,T)=\int_{\mathcal C\{x,0; T \}}  f(\xi(T))
 e^{-\int_{0}^{T}\! V(\xi(\tau),T-\tau) d\tau }
\mathcal{D}_{\xi},
\end{equation}
where the integration is done over the set~${\mathcal C\{x,0; T \}}$ of all
continuous paths $\xi(T): [0,T]\to \mathbb{R}$
from the Banach space $\Xi([0,T], \mathbb{R})$
starting at $\xi(0)=x$ and stopping at arbitrary endpoints at time~$T$. 
$\mathcal{D}_{\xi}$ is the Wiener measure,
and $\xi(t)$ is the Wiener process~\cite{mazo-browmot-2002, nels-browmot-2001}.
One of the advantages of the formulation \eqref{FKformulaF} is that it can be directly applied 
for the unbounded domain without any additional (artificial) boundary conditions.

Path integral~\eqref{FKformulaF}
corresponding to the Wiener process 
is typically approximated by a \emph{finite}
multidimensional integral with the Gaussian measure 
(details are given in Section~\ref{Discretization}). 
The main drawback is that this integral is a high-dimensional one 
and its computation requires a special treatment. 
Several approaches have been developed to compute the multidimensional integrals efficiently. 
The sparse grid method~\cite{smolyak-1963, gerstner-sparsequad-2003} has 
been applied to the computation 
of path integrals in~\cite{gerstner-sparsegrid-1998}, 
but only for dimensions up to $\sim 100$, which is not enough in some applications. 
The main disadvantage of the Monte Carlo simulation is that
it does not allow to achieve a high accuracy~\cite{makri-mc-1987, tretiak-mc-2013}
for some cases (highly oscillatory functions, functions of sum of all arguments).
 
The multidimensional \emph{integrand} can be represented numerically as a
multidimensional array (\emph{a tensor}), which contains values of a multivariate
function on a fine uniform grid. 
For the last decades several approaches have been developed
to efficiently work with tensors. They are based on the idea of \emph{separation of 
variables}~\cite{dela-survey-2009, smilde-book-2004,
khor-lectures-2010, khor-survey-2011}
firstly introduced in~\cite{hitchcock-sum-1927, hitchcock-rank-1927}.
It allows to present a tensor in the \emph{low-rank} or \emph{low-parametric}
format~\cite{treview-2012,kolda-review-2009, lghw-httt-2011},
where the number of parameters used for the approximation is
almost linear (with respect to dimensionality).
To construct such decompositions numerically 
the very efficient algorithms have been developed  recently:
two-dimensional \emph{incomplete cross 
approximation}\footnote{Because the low-rank 
representation of large matrices based on the adaptive cross approximation 
is directly related to the manuscript we summarize the basics of the method
in~\ref{CrossApp}.} for the skeleton decomposition,
three-dimensional cross approximation~\cite{ost-tucker-2008} for 
the Tucker format~\cite{tucker-factor-1966, tucker-factor-1963, 
tucker-factor-1964,lathauwer-svd-2000} in 3D,
\emph{tt-cross}~\cite{ot-ttcross-2010} approximation for the tensor train 
decomposition~\cite{ot-tt-2009, osel-tt-2011},
which can be also considered as a particular case of the
hierarchical Tucker format~\cite{hk-ht-2009, gras-hsvd-2010, hackbusch-2012}
for higher dimensional case.
For certain classes of functions commonly used in the computational physics
(multiparticle Schr\"odinger operator~\cite{mpi-chem3d-2009,
khor-ml-2009, khorombv-qch-2015, schn-dft-2009, schn-dft-2007,khor-chem-2011}, 
functions of a discrete elliptic operator~\cite{kazkh-laplace-2012,khor-low-rank-kron-P1-2006, 
khor-low-rank-kron-P2-2006, beylkin-2002, GHK-ten_inverse_ellipt-2005,
khor-prec-2009},
Yukawa, Helmholtz and Newton potentials~\cite{khor-tuckertype-2007, 
hh-jc-2007,hackbra-expsum-2005,om-stockes-2010}, etc.) there exist
low-parametric representations in separated formats 
and explicit algorithms~\cite{osel-constr-2013,beylkin-high-2005} to obtain
and effectively work with them
(especially \emph{quantized tensor train} (QTT) 
format~\cite{DKhOs-parabolic1-2012, khos-pde-2010, khos-dmrg-2010,
osel-2d2d-2010,lebedeva-tensornd-2011,
 khos-qtt-2010,sav-rank1-2011pre, khor-qtt-2011}). 
In many cases it is very effective to compute the multidimensional 
integrals~\cite{khkhschn-dft-2013} 
using separated representations~\cite{ballani-integrals-2012},
particularly for multidimensional convolutions~\cite{ro-conv3d-2015, 
khor-acc-2010, khkaz-conv-2011pre, dks-ttfft-2012} and 
highly oscillatory functions~\cite{ beylkin-expsumrev-2010}.

Our approach presented here is based on the \emph{low-rank approximation}
of matrices used in an essentially different manner.
We formulate the Feyman-Kac formula as
an iterative sequence of convolutions 
defined on grids of diminishing sizes.
This is done in Section~\ref{MultidimIntegrConvolutions}. 
To reduce the complexity of this computation,
in~Section~\ref{LowRankTheorems} we find a \emph{low-rank} basis set by
applying the \emph{cross approximation} (see~\ref{CrossApp}) to a matrix 
constructed from the values of a one-dimensional
function on a very large grid.
That gives reduce of computational time and memory requirements,
resulting in fast and efficient algorithm presented in Section~\ref{AlgorithmSect}. 
The numerical examples are considered in Section~\ref{NumCalcs}.
The most interesting part is that we are able to treat 
non-periodic potentials without any artificial boundary 
conditions (Section~\ref{ImpurPeriodicExmpl}).

\section{Problem statement}

\subsection{Time discretization }
\label{Discretization}
Equation~\eqref{FKformulaF} corresponds to the Wiener process. 
A standard way to discretize the path integral is 
to break the time range $[0,T]$ into $n$ intervals by points
\begin{equation*}
\tau_k=k \cdot \delta t, \qquad 0\le k\le n,
\qquad n\!: \tau_n=T.
\end{equation*}
The average path of a Brownian particle $\xi(\tau_k)$ after $k$ steps
is defined as
\begin{equation*}
\xi^{(k)}=\xi(\tau_k)=x + \xi_1 + \xi_2 + \ldots + \xi_k,
\end{equation*}
where every random step $ \xi_i$, $1\le i\le k$, is
independently taken from a normal 
distribution~$\mathcal{N}(0,2\sigma \delta t)$
with zero \emph{mean} and \emph{variance} equal to $2 \sigma\delta t$.
By definition, $\xi^{(0)}=x$.

Application of a suitable quadrature rule on the uniform
 grid (i.e., trapezoidal or Simpson rules) 
with the weights~$\{w_i\}_{i=0}^{n}$ to the time integration 
in~\eqref{FKformulaF} gives
\begin{equation}
\label{Vin}
\Lambda(T)=
\int_{0}^{T} V (\xi(\tau), T-\tau) d\tau \approx
\sum_{i=0}^{n} w_{i} V^{(n)}_{i} \delta t,
\qquad
V^{(n)}_{i}\equiv V (\xi(\tau_i),  \tau_{n-i}),
\end{equation}
and transforms the exponential factor to the approximate expression
\begin{equation*}
e^{-\Lambda (T) } \approx
\prod_{i=0}^{n} e^{ -w_{i} V^{(n)}_{i} \delta t }.
\end{equation*}
The Wiener measure, in turn, transforms to
the ordinary $n$-dimensional measure
\begin{equation*}
\mathcal{D}^{(n)}_{\xi}=
\left( \frac{\lambda}{\pi} \right)^{\frac{n}{2}}
\prod_{k=1}^{n}
e^{-\lambda \xi_k^2} \, d \xi_k ,
\qquad
\lambda = \frac{1}{4 \sigma \delta t},
\end{equation*}
and the problem reduces to
an $n$-dimensional integral over the Cartesian coordinate space.
Thus, we can approximate the exact solution~\eqref{FKformulaF}
by $u^{(n)}(x,t)$
\begin{equation*}
u(x,T)=\lim_{n \to \infty} u^{(n)}(x,T)
\end{equation*}
written in the following compact form
\begin{equation}
\label{discr_uv}
u^{(n)}(x,T)=\int_{-\infty}^{\infty} \mathcal{D}^{(n)}_{\xi}
\, f\!\left( \xi^{(n)}\right) \! \prod_{i=0}^{n}
e^{ -w_{i} V^{(n)}_{i} \delta t }.
\end{equation}
The integration sign here denotes an $n$-dimensional integration
over the particle steps $\xi_k$,
and $V^{(n)}_{i}$ is defined in~\eqref{Vin}.
The convergence criterion  in terms of~$n$ for the sequence~\eqref{discr_uv}
is discussed and proven in~\cite{chd-pathint-2001}, p.~33.
The limit of~\eqref{discr_uv} exists if it is a Cauchy sequence.

Our goal is to compute the integral~\eqref{discr_uv}
numerically in an efficient way.

\section{Computational technique}

\subsection{Notations}
\label{SectDefTech}
In this paper vectors (written in columns) are denoted 
by boldface lowercase letters, e.g., $\mathbf{a}$,
matrices are denoted by boldface capital letters, e.g., $\mathbf{A}$.
The $i$-th element of  a vector $\mathbf{a}$ is denoted by $a_i$,
the element $(i,j)$ of a matrix $\mathbf{A}$ is denoted by $A_{ij}$.
A set of vectors $\mathbf{a}_m$, $m_0 \le m \le m_1$ is denoted by
$\left\{ \mathbf{a}_m \right\}_{m=m_0}^{m_1}$,
and the $i$th element of a vector $\mathbf{a}_m$ is denoted
by $a_{mi}$.

\newdefinition{Def}{Definition} 

\begin{Def}
\label{Def2Conv}
Let $\mathbf{a} \in \mathbb{R}^{k}$ and $\mathbf{b} \in \mathbb{R}^{m}$
be vectors and $k \ge m$.
We say that a vector $\mathbf{c} \in \mathbb{R}^{m+k - 1}$ is a \emph{convolution} 
of two ordered vectors $\mathbf{a}$ and $\mathbf{b}$ and write
\begin{equation*}
\mathbf{c} = \mathbf{a} \circ \mathbf{b},
\end{equation*}
if~$\mathbf{c}$ has the following components
\begin{equation*}
c_i=\sum_{j=0}^{m-1}a_{i+j} b_{j}, \qquad
a_i=0, \forall i : \{i \  | \  i < 0 \vee i >= k \}.
\end{equation*}
\end{Def}
The computation of the convolution can be naturally 
carried out as a multiplication by the \emph{Hankel matrix}.
\begin{Def}
\label{Def1HankelGen}
We say that the \emph{Hankel matrix} $\mathbf{A} \in \mathbb{R}^{k\times k}$ is 
generated by row~$\mathbf{a}^T\in \mathbb{R}^{k}$
and column~$\mathbf{b} \in \mathbb{R}^{k-1}$, and denote this by 
\begin{equation*}
\mathbf{A}=[\mathbf{a}^T, \mathbf{b}]_{H},
\end{equation*}
if
\begin{equation*}
\mathbf{A}=
\begin{pmatrix}
a_0 & a_1 & a_2 & \cdots & a_{k-2} & a_{k-1} \\
a_1 & a_2 & a_3 &\cdots & a_{k-1} & b_{0} \\
a_2 & a_3 & a_4 &\cdots & b_{0} & b_{1} \\
\vdots & \vdots & \vdots & \ddots & \vdots & \vdots \\
a_{k-2}& a_{k-1} & b_0 & \cdots & b_{k-4} & b_{k-3}\\
a_{k-1}& b_0 & b_1& \cdots & b_{k-3} & b_{k-2}
\end{pmatrix},
\end{equation*}
\begin{equation}
\label{abshortDef}
 \mathbf{a}^T=\left(a_0, a_1 ,a_2,\ldots, a_{k-2},a_{k-1} \right), 
\qquad
 \mathbf{b}=\left(b_0, b_1, b_2,\ldots, b_{k-2} \right)^T.
\end{equation}
\end{Def}
This compact notation will be used to compute convolutions (when
they are written as a Hankel matrix-vector products). 
As it can be directly verified, $\forall \alpha \in \mathbb{R}$
\begin{equation}
\label{Aalpfab}
\alpha \cdot \mathbf{A} = 
[\alpha \cdot \mathbf{a}^T, \alpha \cdot \mathbf{b}]_{H}.
\end{equation}
\begin{Def}
For two vectors $\mathbf{a}$ and $\mathbf{b}$ from~\eqref{abshortDef}
for the case $a_i=b_i$, $\forall i : 0 \le i < k-1$,
we will also write 
\begin{equation}
\label{abColDef}
\mathbf{a} = 
\begin{pmatrix}
\mathbf{b} \\ a_{k-1}
\end{pmatrix}.
\end{equation}
This notation will be used when vector $\mathbf{b}$ is a subvector of~$\mathbf{a}$.
\end{Def}

\subsection{Multidimensional integration via the sequence of one-dimensional convolutions}
\label{MultidimIntegrConvolutions}
Multidimensional integral~\eqref{discr_uv} can be represented in terms of
an iterative sequence of one-dimensional convolutions.
Indeed, for a one-dimensional function
$F^{(n)}_k(x)$, such that
\begin{equation}
\label{Integral_Iter_k}
F^{(n)}_k(x)= 
\sqrt{\frac{\lambda}{\pi}}
\int_{-\infty}^{\infty} \Phi^{(n)}_{k+1}(x + \xi)  
\, e^{-\lambda \xi^2 } d\xi,
\qquad     x \in \mathbb{R},
\qquad   k=n, n-1, \ldots, 1,
\end{equation}
with
\begin{equation}
\label{Phi_iter_k}
\Phi^{(n)}_{k+1}(x)= 
F^{(n)}_{k+1}(x) \,
e^{ -w_k V (x, \tau_{n - k}) \delta t  },
\end{equation}
and the initial condition
\begin{equation}
\label{Init_F_n}
F^{(n)}_{n+1}(x)=f(x),
\end{equation}
the solution~\eqref{discr_uv} reads
\begin{equation}
\label{un1xtFNumer}
u^{(n)}(x,T) = F^{(n)}_{1}(x) \, e^{-w_0 V(x, T) \, \delta t}.
\end{equation}
The iteration starts from $k=n$ and goes down to $k=1$.
Since the function $\Phi^{(n)}_{k}(x)$ is bounded and the convolution~\eqref{Integral_Iter_k}
contains the exponentially decaying Gaussian, 
the integral has finite lower and upper bounds.
Consider
\begin{equation}
\label{short_Integral_Iter_k}
F^{(n)}_k(x)\approx
\tilde F^{(n)}_k(x)= 
\sqrt{\frac{\lambda}{\pi}}
\int_{-a_x}^{a_x-h_x} \tilde \Phi^{(n)}_{k+1}(x + \xi)  
\, e^{-\lambda \xi^2 } d\xi.
\end{equation}
We suppose that the product
$ \Phi^{(n)}_{k+1}(x + \xi)  \, e^{-\lambda \xi^2 }$
rapidly decays, so that for $a_x$ large enough, we can
approximate the integral $ F^{(n)}_k(x)$  in~\eqref{Integral_Iter_k} 
by $\tilde F^{(n)}_k(x)$ and
assume that this approximation has an error~$\varepsilon$ in some norm
\begin{equation*}
\left\| F^{(n)}_k(x) - \tilde F^{(n)}_k(x) \right\| < \varepsilon.
\end{equation*}
%%%%%%%%%%%%%%%%%%%%%%%%%%%%%%%%%%%%%%%%%%%%%%%%
%Figure 1
\begin{figure}
\includegraphics{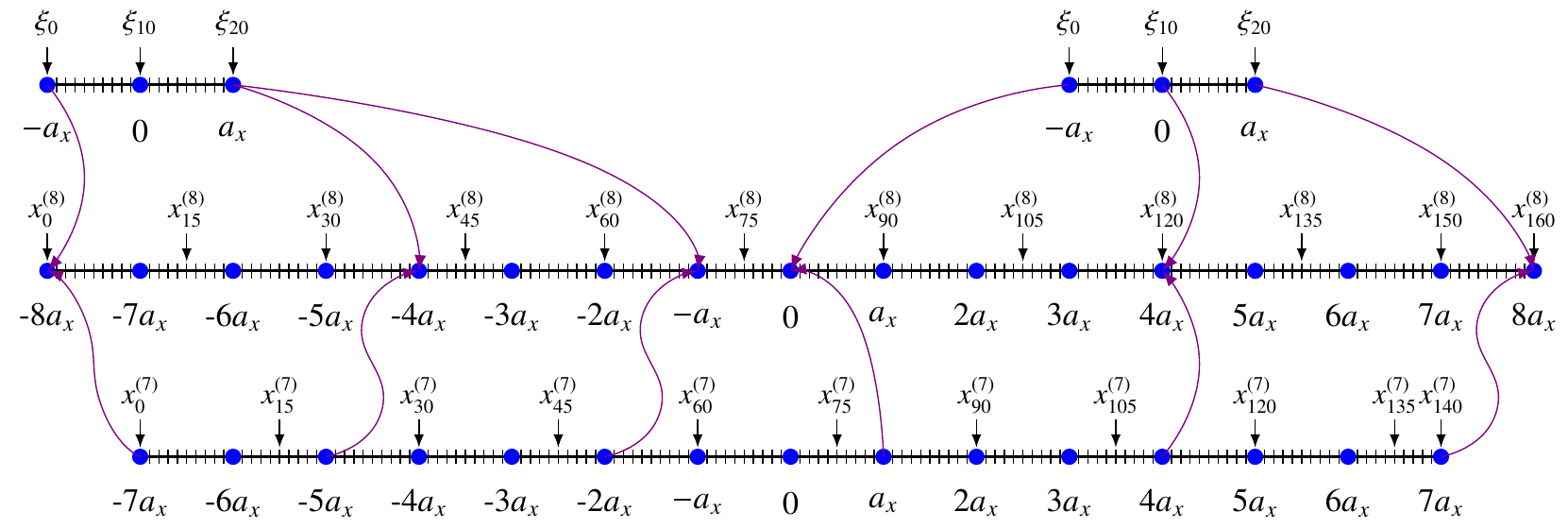}
\caption{
\label{PictMesh01}
A correspondence of meshes for two nearest iterations $x_{i+j}^{(k+1)}=x^{(k)}_{i} + \xi_{j}$
from equation~\eqref{yxksiij} for $k=7$. Blue filled circles separate the ranges corresponding to 
different steps $m$, $1 \le m \le k$ in time $[-m a_x, m a_x)$. Ticks on the axes label the mesh points. 
Violet curved lines show correspondence~\eqref{yxksiij} between 
the two meshes for nearest iterations. 
}
\end{figure}
% Figure 1
%%%%%%%%%%%%%%%%%%%%%%%%%%%%%%%%%%%%%%%%%%%%%%%%
This approximation has an important drawback:
as soon as $F^{(n)}_1(x)$ has to be computed on the semi-open interval $[-a_x, a_x)$,
the domain of $F^{(n)}_n(x)$ should be taken larger, i.e.~$[-n a_x, n a_x)$ for $n$~steps, 
because of the convolution structure of the 
integral~\eqref{short_Integral_Iter_k}.
Indeed, if we suppose, that the function~$F^{(n)}_k(x)$ is computed on the uniform mesh
\begin{equation}
\label{xunifmeshM}
x^{(k)}_i= -ka_x + i h_{x}, \qquad
0 \le i < kM, \qquad
h_x=a_x / N_{x}, \qquad
M=2N_{x},
\end{equation}
and the integration mesh is chosen to be nested in~\eqref{xunifmeshM}
with the \emph{same step} $h_{x}$
\begin{equation}
\label{ksiunifmesh}
\xi_j= -a_{x} + j h_{x}, 
\qquad
0 \le j < M,
\end{equation}
then the function $F^{(n)}_{k+1}(x)$ is defined on the mesh
\begin{equation}
\label{ximesh}
x^{(k+1)}_i= -(k+1) a_x + i h_{x}, \qquad
0 \le i < (k+1) M,
\end{equation}
and
\begin{equation}
\label{yxksiij}
x^{(k+1)}_{i+j}  =x^{(k)}_i+\xi_j.
\end{equation}
The last equality follows from definitions~\eqref{xunifmeshM} and~\eqref{ksiunifmesh}.
This is illustrated in Figure~\ref{PictMesh01}.

The integral~\eqref{short_Integral_Iter_k} can be calculated 
for every fixed $x^{(k)}_i$ of the mesh~\eqref{xunifmeshM}
as the quadrature sum with the weights $\{\mu_j\}_{j =0}^{M-1}$ 
\begin{equation}
\label{Iksum}
\tilde F^{(n)}_k\left(x^{(k)}_i\right) \approx
\sum_{j=0}^{M-1} \mu_j
\tilde \Phi^{(n)}_{k+1}\left(x^{(k+1)}_{i+j}\right)
p(\lambda,\xi_j), \qquad
p(\lambda,\xi)=\sqrt{\frac{\lambda}{\pi}}e^{-\lambda \xi^2}
\end{equation}
\begin{equation}
\label{tilde_Phi_iter_k}
\tilde \Phi^{(n)}_{k+1}\left(x^{(k+1)}_i\right)= 
\tilde F^{(n)}_{k+1}\left(x^{(k+1)}_i\right) \,
e^{ -w_k V (x^{(k+1)}_i, \tau_{n - k}) \delta t  }
\end{equation}
The complexity of the computation of $\tilde F^{(n)}_k\left(x^{(k)}_i\right)$ for all $i$
is $\mathcal{O}(k N_x^2)$ flops. It can be 
reduced to $\mathcal{O}(k N_x \log N_x)$
by applying the Fast Fourier transform (FFT) for convolution~\eqref{Iksum}.
Full computation of $\tilde F^{(n)}_1\left(x^{(1)}_i\right)$ costs $\mathcal{O}(n^2 N_x \log N_x)$
operations and $\mathcal{O}( n N_x )$ floating-point numbers.
This complexity becomes prohibitive for large $n$ (i.e., for small time steps), but can be reduced.
Below we present a fast approximate method for the calculation 
of~$\tilde F^{(n)}_1\left(x^{(1)}_i\right)$ in $\mathcal{O}(nr N_x \log N_x + nr^2 N_x)$ flops
and $\mathcal{O}(r N_x)$ memory cost with $r \ll n$,
by applying low-rank decompositions.

\subsection{Low-rank basis set for the convolution array.}
\label{LowRankTheorems}
In this section we provide a theoretical justification for our approach. 
Consider a sequence of matrices $\mathbf{A}^{(k)} \in \mathbb{R}^{kM \times M}$ 
corresponding to the iterative process~\eqref{Iksum}
and constructed in the following way
\begin{equation}
\label{ADirectDef}
A^{(k)}_{ij}=a^{(k)}_{i+j}\equiv
\tilde \Phi^{(n)}_{k}\left(x^{(k)}_{i+j}\right),
\qquad 0 \le j < M,
\qquad 0 \le i < kM,
\end{equation}
where $k$ is the iteration number.

Let us now consider iteration~\eqref{Iksum} at the step $k = k_0$ and 
for simplicity omit the index $k_0+1$ in
the matrix and mesh notations~\eqref{ADirectDef}.
Let us also denote the sum~\eqref{Iksum} 
for $x_i$ taken from the grid~\eqref{xunifmeshM} by $s_i=\tilde F^{(n)}_{k_0}\left(x_i\right)$ and 
set $p_j \equiv \mu_j p(\lambda, \xi_j)$. Then 
\begin{equation}
\label{pSiDefNote}
s_i= \sum_{j=0}^{M-1} 
A_{ij}\ p_j \qquad \Leftrightarrow \qquad
\mathbf{s} = \mathbf{A} \mathbf{p}.
\end{equation}
The equality~\eqref{pSiDefNote} establishes the recurrence relation between
iterations at the step~$k$ (the right-hand side) 
and the step~$k-1$ (the left-hand side) according to~\eqref{Iksum} 
and~\eqref{tilde_Phi_iter_k}, see Figure~\ref{Fig2SuperA}.
%%%%%%%%%%%%%%%%%%%%%%%%%%%%%%%%%%%%%%%%%%%%%%%%%%%%%%%
% Figure 2
\begin{figure}
\includegraphics{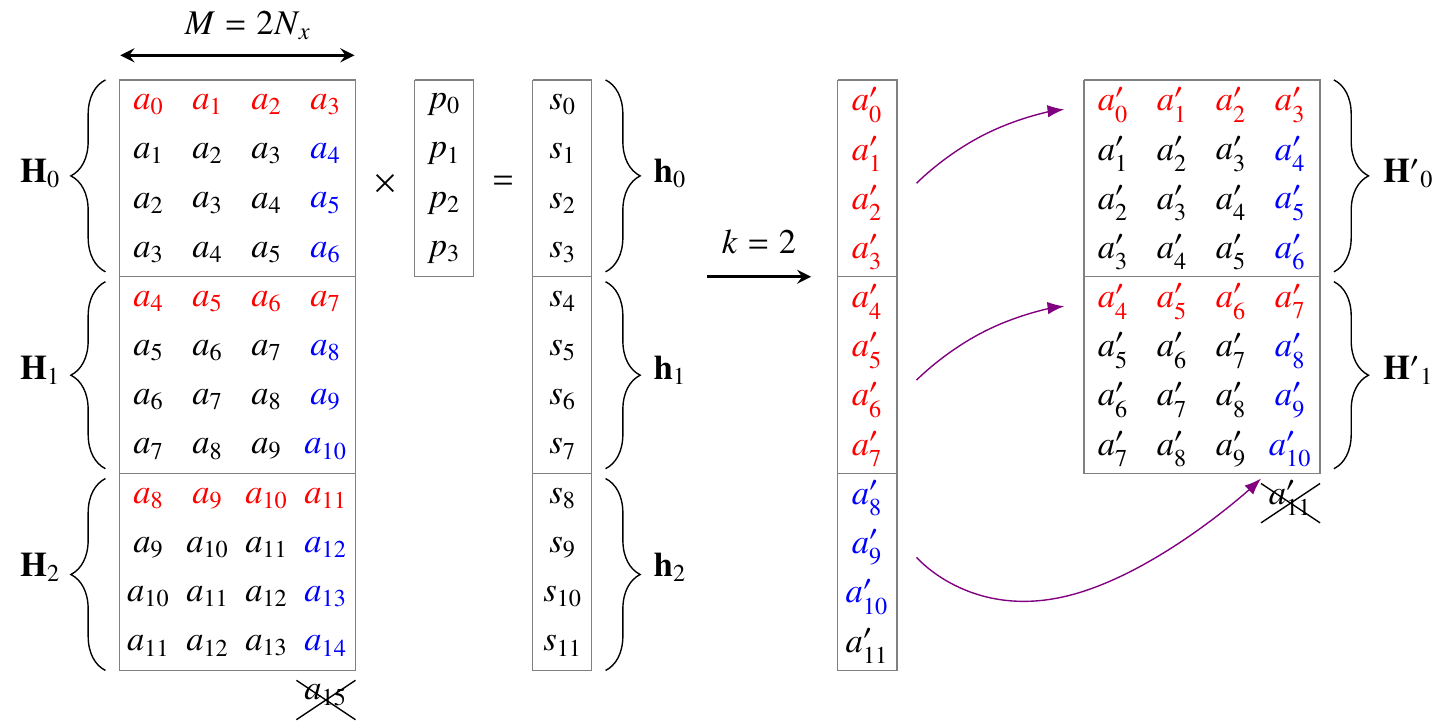}
\caption{
\label{Fig2SuperA}
Transition between two neighbour iterations is illustrated.
The left-hand-side matrix~$\mathbf{A}$ is multiplied 
by vector~$\mathbf{p}$ in a resulting vector~$\mathbf{s}$ 
according to~\eqref{pSiDefNote}. Explicit structure of matrix 
blocks~$\mathbf{H}_m$ in~\eqref{AdefHk} and 
vector blocks~$\mathbf{h}_m$ in \eqref{sHp} is shown.
Then entries of vector~$\mathbf{s}$ are multiplied by
corresponding factor~$e^{-w_k V(x_i,\tau_{n-k})\delta t}$
to produce the next iteration step~\eqref{Phi_iter_k}.
From a new vector~$\mathbf{a}'$ there formed
a new matrix~$\mathbf{A}'$ according to~\eqref{ADirectDef}.
The last point from the previous iteration is not needed
and is thrown out. Then the steps repeat for the next iteration.
}
\end{figure}
%%%%%%%%%%%%%%%%%%%%%%%%%%%%%%%%%%%%%%%%%%%%%%%%%%%%%%%

The matrix $\mathbf{A}$ is a Hankel matrix,
as follows from definition~\eqref{ADirectDef},
and consists of $k$ square blocks $\mathbf{H}_{m}$,
$0 \le m < k$, such that
\begin{equation}
\label{AdefHk}
\mathbf{A}^{\!T}=\left(\mathbf{H}_0, \mathbf{H}_1 \ldots \mathbf{H}_{k-1} \right).
\end{equation}
Here, every block $\mathbf{H}_m$ is a Hankel matrix as well \emph{generated} by
the upper row $\mathbf{l}^{T}_{m}$ and the right column 
$\mathbf{r}_{m+1}$ correspondingly:
\begin{equation*}
%\label{HmDef}
\mathbf{H}_m=\left[ \mathbf{l}^T_{m}, \mathbf{r}_{m+1} \right]_H,
\end{equation*}
(the notation $[\mathbf{a}, \mathbf{b} ]_{H}$ is 
introduced in Section~\ref{SectDefTech}, Definition~\ref{Def1HankelGen}),
where
\begin{equation}
\label{lrDef}
\begin{split}
\mathbf{l}^{T}_{m}&= (a_{i_0}, a_{i_0+1}, \ldots, a_{i_0+M-2}, a_{i_0+M-1}), \\
\mathbf{r}^{T}_{m}&= (a_{i_0}, a_{i_0+1}, \ldots, a_{i_0+M-2}),
\end{split}
\qquad i_0 = m \cdot M,
\qquad 0\le m < k,
\end{equation}
by the definition, see Figure~\ref{Fig2SuperA}.
It can also be represented as a sum of two 
anti-triangular\footnote{By anti-triangular matrix we call a matrix which is
triangular with respect to the anti-diagonal of the matrix.} Hankel matrices  
\begin{equation}
\label{HLRsum}
\mathbf{H}_m = \mathbf{L}_m+ \mathbf{R}_m, \qquad
\mathbf{L}_m=\left[\mathbf{l}^T_{m\cdot M}, \mathbf{0} \right]_{H}, \qquad
\mathbf{R}_m=\left[ \mathbf{0}^T, \mathbf{r}_{(m+1)\cdot M} \right]_H,
\end{equation}
where the upper-left $\mathbf{L}_m$ has nonzero anti-diagonal and 
the bottom-right $\mathbf{R}_m$ has zero anti-diagonal according to~\eqref{lrDef}.

Equation~\eqref{pSiDefNote} may be rewritten in the block form
(see again Figure~\ref{Fig2SuperA})
\begin{equation}
\label{sHp}
\mathbf{h}_m = \mathbf{H}_m \mathbf{p}, \qquad 
\mathbf{h}^T_m=\left( s_{m\cdot M}, s_{m\cdot M +1} \ldots s_{m\cdot M + M -1} \right).
\end{equation}
Here every block $\mathbf{H}_m$ is multiplied by 
the same vector~$\mathbf{p}$.
The number of matrix-vector multiplications can be reduced,
if the dimension $d$ of the linear span $\mathcal{H}=\left\{ \mathbf{H}_m \right\}_{m=0}^{k-1}$
is less than~$k$. 
Before estimation of the dimension we formulate some auxiliary lemmas
(proven in the~\ref{App:proveLem}).

\newtheorem{Lem}{Lemma}
\newdefinition{Rem}{Remark} 
\newtheorem{Th}{Theorem}
%
% lemma 1
\begin{Lem}
\label{lem2}
Let $\{ \mathbf{u}_{i} \}_{i=0}^{r_1-1}$ be a basis set of span $\{ \mathbf{l}_{m} \}_{m=0}^{k-1}$,
 $r_1 \le k$, and let matrix $\mathbf{U}_i = \left[\mathbf{u}^T_{i}, \mathbf{0} \right]_{H}$.
Then $\{ \mathbf{U}_{i} \}_{i=0}^{r_1-1}$ is a basis set 
of span $\{ \mathbf{L}_{m} \}_{m=0}^{k-1}$ from~\eqref{HLRsum}.
\end{Lem}
% lemma 2
\begin{Lem}
\label{lem3}
Let $\{ \mathbf{w}_{i} \}_{i=0}^{r_2-1}$ be a basis set of span $\{ \mathbf{r}_{m} \}_{m=1}^{k}$,
$r_2 \le k$, and let matrix $\mathbf{W}_i = \left[\mathbf{0}^T, \mathbf{w}_{i} \right]_{H}$.
Then $\{ \mathbf{W}_{i} \}_{i=0}^{r_2-1}$ is a basis set 
of span $\{ \mathbf{R}_{m} \}_{m=0}^{k-1}$ from~\eqref{HLRsum}.
\end{Lem}
% lemma 3
\begin{Lem}
\label{lem4}
Let $\{ \mathbf{u}_{i} \}_{i=0}^{r_1-1}$ be a basis set of span 
$\left\{ \mathbf{l}_{m} \right\}_{m=0}^{k}$, 
such that $\mathbf{u}_{i}^T=(\mathbf{w}_i^T, u_{i,(M-1)})$
according to~\eqref{abColDef}.
Then $\{ \mathbf{w}_{i} \}_{i=0}^{r_1-1}$ is a basis set of
span $\{ \mathbf{r}_{m} \}_{m=1}^{k}$.
\end{Lem}
Let us define a basis set $\left\{\mathbf{Q} \right\}_{i=0}^{2r-1}$ as follows
\begin{equation}
\label{QDefUW}
\mathbf{Q}_i=
\begin{cases}
\mathbf{U}_i, & 0 \le i < r \\
\mathbf{W}_{i - r}, & r \le i < 2r
\end{cases}
\end{equation}
%
% Theorem 1
%
An obvious corollary of the previous Lemma is the following Theorem.
\begin{Th}
\label{TheoremBasisRank}
The dimension of the linear span of matrices $ \left\{\mathbf{H}_m \right\}_{m=0}^{k-1}$ is equal to $2r$.
Moreover, it is contained in the linear span of the matrices $\left\{\mathbf{Q}_i \right\}_{i=0}^{2r-1}$ defined in \eqref{QDefUW}.
\end{Th}
\begin{proof}
The matrix $\mathbf{H}_m$ can be written as a sum~\eqref{HLRsum}, 
$\mathbf{H}_m=\mathbf{L}_m+\mathbf{R}_m$. 
According to \emph{Lemma~\ref{lem2}}, the set $\left \{\mathbf{U}_i \right\}_{i=0}^{r-1}$ is 
a basis set of the span $\left \{ \mathbf{L}_m \right \}_{m=0}^{k-1}$.
By \emph{Lemma~\ref{lem4}},
$\left\{\mathbf{w}_i\right\}_{i=0}^{r-1} $ is a basis set of the span
$\left\{\mathbf{r}_{m}\right\}_{m=0}^{r-1} $, and by
\emph{Lemma~\ref{lem3}},
$\left\{ \mathbf{W} \right\}_{i=0}^{r-1}$ 
is a basis set of 
$\left\{ \mathbf{R} \right\}_{m=0}^{k-1}$. 
The subspaces $\left\{ \mathbf{U}_i \right\}_{i=0}^{r-1}$  and  
$\left\{ \mathbf{W}_i \right\}_{i=0}^{r-1}$  contain only zero matrix in common, 
so the dimension of the basis is~$2r$.
\end{proof}
%
% Lemma 4
\begin{Lem}
\label{Lem4FFT}
Let $\{ \mathbf{u}_{i} \}_{i=0}^{r-1}$ be a basis set of span 
$\left\{ \mathbf{l}_{m} \right\}_{m=0}^{k}$.
Then for basis matrices
$\left\{\mathbf{Q}_i \right\}_{i=0}^{2r-1}$ defined in~\eqref{QDefUW}
the computation of the matrix-by-vector products 
\begin{equation}
\label{ktBas}
\mathbf{k}_i=\mathbf{U}_{i}\mathbf{p}, \qquad 
\mathbf{t}_i=\mathbf{W}_{i}\mathbf{p},
\end{equation}
costs $\mathcal{O}(M\log M)$ flops for a fixed $0 \le i < r$.
\end{Lem}
\begin{proof}
Consider a Hankel matrix
\begin{equation*}
\mathbf{G}_i=
\begin{pmatrix}
\mathbf{W}_{i} \\
\mathbf{U}_{i}
\end{pmatrix}.
\end{equation*}
A product $\mathbf{G}_i \mathbf{p}$
is a result of the convolution $\mathbf{u}_{i} \circ \mathbf{\hat p}$,
which can be done by the FFT ~\cite{brigham-fft-1988, nuss-fft-1981}
procedure in $\mathcal{O}(M \log M)$ flops
for a fixed $0 \le i < r$.
The vector $\mathbf{\hat p}=(p_{M-1},\ldots, , p_{1}, p_{0})^T$ is taken in the reverse order.
\end{proof}
Once the basis $\left\{ \mathbf{Q}_i \right\}_{i=0}^{d-1}$ for the span of $\mathcal{H}$ is found, 
the complexity of the multiplication $\mathbf{A}\mathbf{p}$ in~\eqref{pSiDefNote} can be estimated 
as follows.
%
% Theorem 2 
%
\begin{Th}
\label{Theorem2}
Let the set $\left\{ \mathbf{Q}_i \right\}_{i=0}^{2r-1}$ defined in~\eqref{QDefUW} be
a basis set of the linear span $\mathcal{H}$ 
generated by the set of Hankel matrices~$\mathbf{H}_m$ defined in~\eqref{AdefHk}.
Then the computation of any $K_s$ elements $s_i$  of the vector $\mathbf{s}$~\eqref{pSiDefNote}
costs $\mathcal{O}(rM\log M + r^2 M)$ flops for $K_s =\mathcal{O}(r M)$.
\end{Th}
\begin{proof}
Indeed, by the assumption 
$\mathbf{H}_m=\sum_{i=0}^{2r-1} c_{mi} \mathbf{Q}_{i}$
for each $m$, $0\le m < k$.
The complexity of the product $\mathbf{Q}_{i} \mathbf{p}$, $0\le i < 2r$ for a 
fixed~$i$ is $\mathcal{O}( M\log M)$ flops
by Lemma~\ref{Lem4FFT}.
The computation of such products for all $i$
takes  $\mathcal{O}(r M\log M)$ flops.

The vector $\mathbf{h}_m$,
which is a subvector of~$\mathbf{s}$,
is represented via few matrix-by-vector products~\eqref{ktBas}
as follows
\begin{equation}
\label{alphaBetDecomp}
\mathbf{h}_{m}=\mathbf{H}_{m}\mathbf{p}=
\mathbf{L}_{m}\mathbf{p} + \mathbf{R}_{m}\mathbf{p}=
\sum_{i=0}^r  \alpha_{mi} \mathbf{U}_{i}\mathbf{p} + 
\beta_{mi}\mathbf{W}_{i}\mathbf{p}=
\sum_{i=0}^r  \alpha_{mi} \mathbf{k}_{i} +  \beta_{mi}\mathbf{t}_{i}.
\end{equation}
The computation of its $i$th component $h_{mi}$ takes $\mathcal{O}(r)$ flops for any $m$. 
Computation of $\mathcal{O}(r M)$ components $s_j$
of the vector $\mathbf{s}$, which are also the components of
the particular vector $\mathbf{h}_m$
(for $m=\left\lfloor \frac{j}{M} \right\rfloor$),
costs, in turn, $\mathcal{O}(r^2 M)$ flops.
Finally, $\mathcal{O}(rM\log M+r^2 M)$.
\end{proof}
\begin{Rem}

Each component of the resulting vector can be computed by the formula
\begin{equation}
\label{EqForSFin}
s_j=h_{m_j l_j}=
\sum_{i=0}^r  \alpha_{m_ji} k_{i l_j} +  \beta_{m_ji}t_{i l_j},
\qquad
m_j = \left \lfloor \frac{j}{M} \right \rfloor,
\qquad
l_j=j \!\!\! \mod M.
\end{equation}
Here $k_{il_j}$ is the $l_j$-th component of the vector~$\mathbf{k}_i$ and
$t_{il_j}$ is the $l_j$-th component of the vector~$\mathbf{t}_i$.
\end{Rem}
\begin{Rem}
It follows from Lemma~\ref{lem4} that $\alpha_{i+1,j}=\beta_{ij}$  
in~\eqref{alphaBetDecomp}.
\end{Rem}

\subsection{Final algorithm}
\label{AlgorithmSect}
To compute $\tilde F_{1}^{(n)}(x_i)$, which defines the final solution~\eqref{un1xtFNumer}
on the mesh~\eqref{xunifmeshM},
one needs to carry out iterations~\eqref{Iksum} starting from $k=n$ down to $1$.
At each iteration step $k$ we construct a function $f_k\left(x^{(k)}_i\right)$, 
which \emph{approximates}
the entries~$s^{k}_i$ in equation~\eqref{EqForSFin} as follows.
Suppose, that the function $f_{k+1}\left(x^{(k+1)}_i\right)$ has been already  constructed 
at the previous step $k + 1$.
Then, to compute $f_k\left(x^{(k)}_i\right)$ at the current iteration $k$,
we consider\footnote{but do not compute all its elements}
the matrix~$\mathbf{\Phi}^{(k+1)}$ with the entries
\begin{equation}
\label{FmatrixCross}
\Phi^{(k+1)}_{ij}= f_{k+1}\left(y_{ij}\right) e^{-w_{k}V(y_{ij},\tau_{n-k})\delta t}, \qquad
y_{ij}=x^{(k+1)}_{i +  j \cdot M},
\end{equation}
and apply the \emph{cross approximation}~\eqref{CrossDecomp} to this matrix.
The columns of this matrix are vectors $\mathbf{h}^{(k+1)}_m$
element-wise multiplied by the corresponding exponential factor with the potential~\eqref{FmatrixCross},
see Figure~\ref{Pict3BasisMatr}.
The algorithm of the cross approximation requires only $\mathcal{O}(rM)$ entries,
which are being chosen adaptively. They are calculated by 
the function $f_{k+1}\left(y_{ij} \right)$
on-the-fly for the particular points $y_{ij}$. Thus,
\begin{equation}
\label{FBVcross}
\mathbf{\Phi}^{(k+1)}=\mathbf{B} \mathbf{V}^T, \qquad 
\mathbf{B} \in \mathbb{R}^{M \times r}, \qquad
\mathbf{V} \in \mathbb{R}^{(k+1) \times r}, \qquad
r \ll M,
\end{equation}
where $\mathbf{B}$ and $\mathbf{V}$ are matrices  of the rank~$r$ saved in memory.
By construction, the $m$-th column of matrix $\mathbf{\Phi}^{(k+1)}$ is 
the vector~$\mathbf{l}_{m}$
from~\eqref{lrDef} and the $i$-th column of matrix $\mathbf{B}$ is the basis vector
$\mathbf{u}_i$ from Lemma~\ref{lem2}. Hence,
$\mathbf{V}^{T}$ is the matrix of coefficients of the decomposition~\eqref{lmuiDecompose}.
Once the cross approximation~\eqref{FBVcross} is obtained,
the memory allocated for all data structures related 
to $f_{k+1}\left(x^{(k+1)}_i\right)$ can be overwritten at the next iteration.
%%%%%%%%%%%%%%%%%%%%%%%%%%%%%%%%%%%%%%%%%%%%%%%%%%%%%%%%%%%%
\begin{figure}
\includegraphics{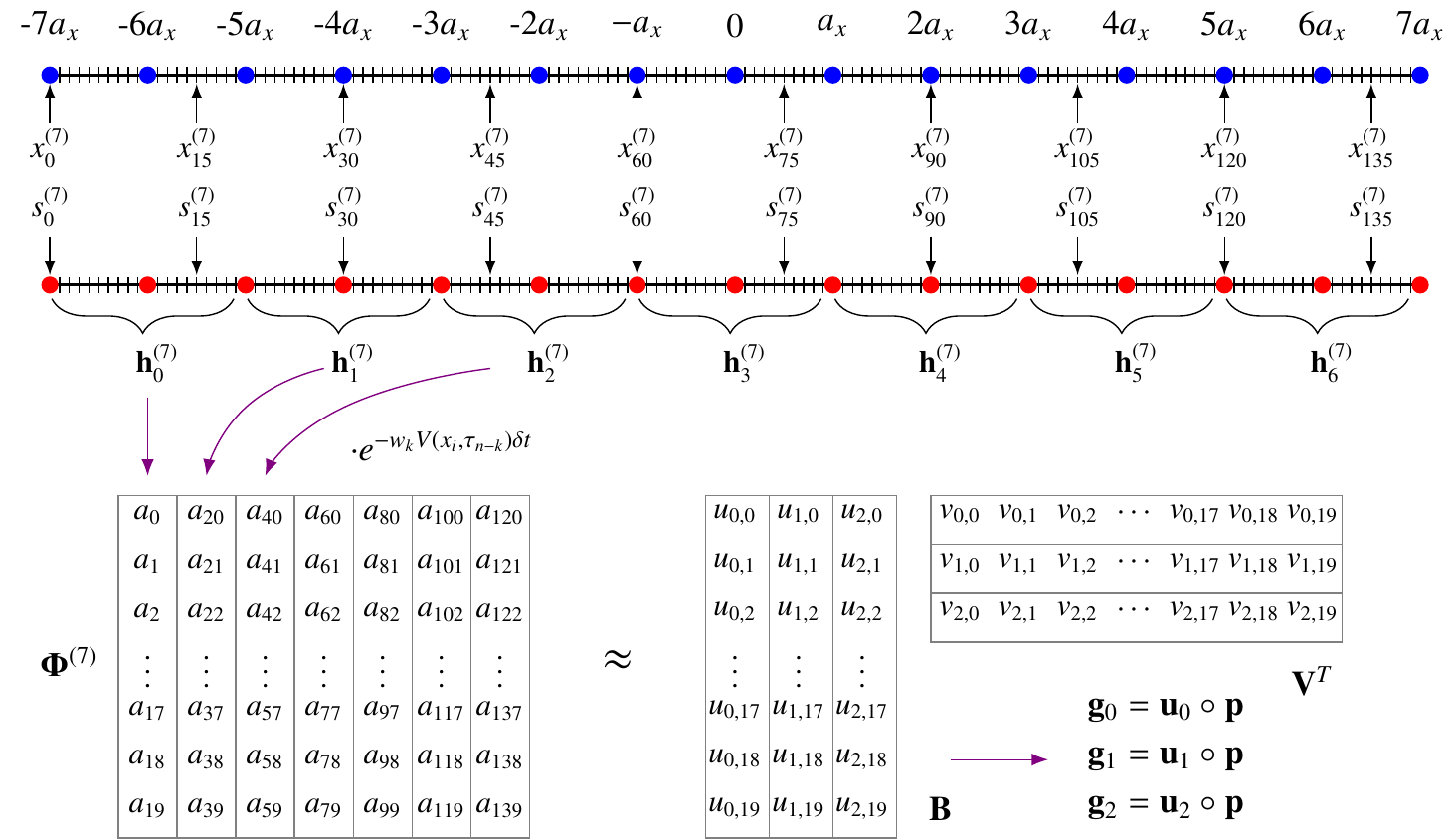}
\caption{
Construction of matrix $\mathbf{\Phi}^{(k)}$ from 
a one-dimensional convolution~\eqref{pSiDefNote} according
to algorithm in Section~\ref{AlgorithmSect}.
On a spatial homogeneous mesh~\eqref{xunifmeshM}
the corresponding entries of vector $\mathbf{s}$~\eqref{pSiDefNote}
are calculated. By definition, vector $\mathbf{s}$ is composed
from vectors~$\mathbf{h}_m$~\eqref{sHp}. 
Each column of the matrix~$\mathbf{\Phi}^{(k)}$
is composed of $\mathbf{h}_m$ multiplied by a corresponding 
factor $e^{-w_k V(x_{i},\tau_{n-k}) \delta t}$.
Then this matrix is decomposed 
by a \emph{cross approximation}
$\mathbf{\Phi}^{(k)}=\mathbf{B} \mathbf{V}^T$~\eqref{FBVcross}. 
For the approximation there needed only some elements
of matrix~$\mathbf{\Phi}^{(k)}$, which are chosen
adaptively and computed on-the-fly.
Then convolutions $\mathbf{g}_i=\mathbf{u}_i \circ \mathbf{p}$ are 
calculated via fast Fourier transform and saved
in the memory. Particular values of $s_i^{(k)}$ for the next iteration step $k-1$ can be computed by formula~\eqref{EqForSFin}  then.
\label{Pict3BasisMatr}
}
\end{figure}

%%%%%%%%%%%%%%%%%%%%%%%%%%%%%%%%%%%%%%%%%%%%%%%%%%%%%%%%%%%%

Computation of the circulant matrix-vector products~\eqref{ktBas} 
is done according to Lemma~\ref{Lem4FFT} by the convolution
$\mathbf{g}_i = \mathbf{b}_{i} \circ \mathbf{\hat p}$, where
$\mathbf{b}_{i}$ is a column of the matrix $\mathbf{B}$.
The vectors $\mathbf{g}_i=(\mathbf{t}_i,\mathbf{k}_i)^{T}$ are also saved in the memory.
Then $f_k\left(x^{(k)}_i\right)$ is calculated by equation~\eqref{EqForSFin},
and the algorithm proceeds to the next iteration.

At some iteration step~$k$ the rank of the decomposition~\eqref{FBVcross}
will reach the number of columns 
and from this iteration it will be more efficient
to carry out the convolution~\eqref{pSiDefNote}
without low-rank approximation.
Complexity of one iteration of the presented algorithm is estimated in Theorem~\ref{Theorem2}.
Finally, for all $n$ steps it is $\mathcal{O}(nrM\log M +nr^2M)$ flops, $r \ll n$.
The standard FFT based algorithm applied to the whole array without
any low-rank compression at each step gives complexity for all $n$ steps equal to
$\mathcal{O}(n^2M\log M)$ flops. We illustrate this theoretical estimations
by the example from Section~\ref{CauchySubsect} in~Figure~\ref{FFT42Cauchy}.
% Figure FFT
\begin{figure}
\includegraphics[width = 10cm]{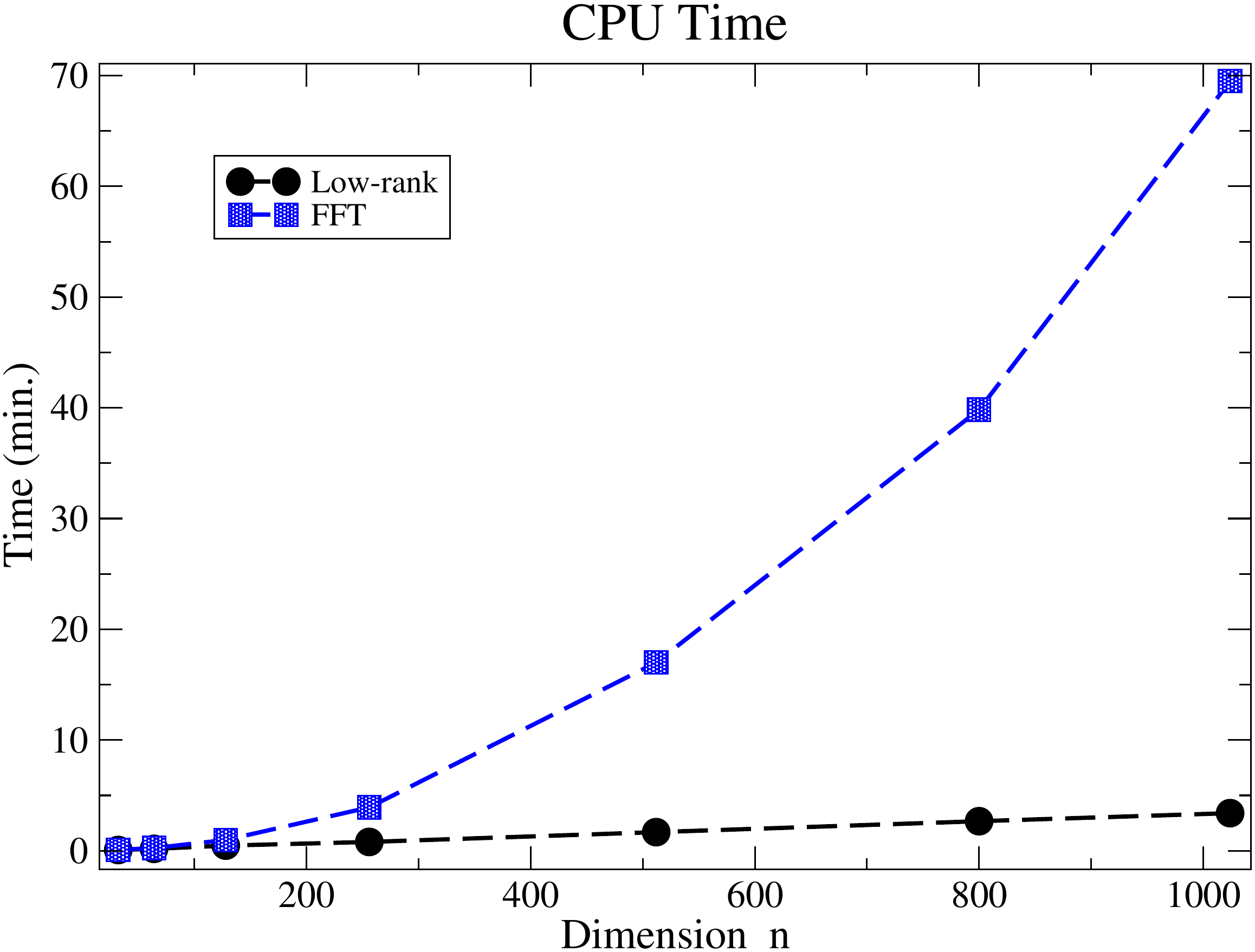}
\caption{
A numerical illustration of theoretical estimations for
the example from Section~\ref{CauchySubsect}.
For a standard FFT based algorithm applied to the whole array 
the $\mathcal{O}(n^2M\log M)$ flops complexity is labeled by square points.
The low-rank complexity $\mathcal{O}(nrM\log M +nr^2M)$ flops is labeled by
circles. The time is scaled in minutes, $n$ is the number of dimensions (iteration steps),
$M=8000$, $r=10$.
\label{FFT42Cauchy}
}
\end{figure}

Basically, the asymptotic complexity, proven in Theorem~\ref{Theorem2}
is practically useful for~$r \ll n$.  This is the main assumption for the 
matrix from~\eqref{FmatrixCross}.
Existence of such an approximation (and the properties of the initial problem) 
is in general still an open question.
Some particular cases were studied in~\cite{demanet-coherent-2013}.
It was shown, that the cross approximation converges for matrices having 
singular vectors satisfying the \emph{coherence} property.
Some estimations can be found in~\cite{bebe-bem-2006, bebe-cheb-2009, bebe-acamf-2011} also.
There is a theoretical idea how to identify the existence of 
the low-rank structure of a given matrix generated by a one-dimensional
target function \emph{a priory} (see~\cite{lubich-integr-2006} 
and \ref{LowRankStrSect} for details).

\section{Numerical experiments and discussions}
\label{NumCalcs}
\subsection{Harmonic Oscillator}
\label{HarmOsc}
As a first example, let us consider a model system, which can be solved analytically,
with the initial condition $f_{ho}(x)$ and the dissipation rate $V_{ho}(x,t)$ defined as
\begin{equation}
\label{OscfV}
f_{ho}(x)=p(\beta,x)=\sqrt{ \frac{\beta}{\pi}} e^{-\beta x^2},
\qquad
V_{ho}(x,t)=\frac{x^2}{t+1}.
\end{equation}
According to equation~\eqref{un1xtFNumer}
the exact solution $u^{(n)}_{ho}(x,t)$ for the particular case~\eqref{OscfV} has
the following form (see~\ref{App:Osc} for derivation)
\begin{equation}
\label{OscExact}
u^{(n)}_{ho}(x,t)=\Psi^{(n)}_{1}(x) \, e^{-w_0 V(x, t) \, \delta t}.
\end{equation}
Comparison of the numerical low-rank solution 
with the exact one~\eqref{OscExact} gives the relative error 
\begin{equation}
\label{RelEpsDef}
\epsilon=
\left \|  \mathbf{\tilde u} - \mathbf{u} \right \| \Big/
\left \|  \mathbf{u} \right \|,
\end{equation}
which in the order of magnitude is equal to the machine precision,
where $\mathbf{\tilde u} $ is an approximate solution on the
final mesh and $\mathbf{u}$  is the exact one on the same mesh.
For our example
\begin{equation*}
\qquad
\tilde u_i =\tilde F^{(n)}_{1}(x_i) \, e^{-w_0 V_{ho}(x_i, T) \, \delta t}, 
\qquad
u_i = \Psi^{(n)}_{1}(x_i) \, e^{-w_0 V_{ho}(x_i, T) \, \delta t}.
\end{equation*}
Here $\sigma=0.25$, $T = 10$, $n = 100$, and the mesh is 
a uniform one on $[-2, 2]$ with $M = 2N_x=8000$ points.
It is interesting that the scheme is exact for this case.

\subsection{Cauchy Distribution}
\label{CauchySubsect}
The second example is taken from~\cite{gerstner-sparsegrid-1998} and 
is interesting because it can be solved analytically as well.
For $V_{c}(x,t)$ and initial condition $f_{c}(x)$ such that
\begin{equation}
\label{CauchyEq}
V_{c}(x,t)=-\frac{1}{t+1} +2 \sigma \frac{3x^2-1}{(x^2+1)^2},
\qquad
f_{c}(x)=\frac{1}{\pi}\frac{1}{x^2+1},
\end{equation}
the exact solution is
\begin{equation*}
u_{c}(x,t)=\frac{1}{\pi}\frac{t+1}{x^2+1}.
\end{equation*}
In Table~\ref{Table1Cauchy} we present  numerical results 
demonstrating the numerical order of scheme by the Runge formula
\begin{equation*}
%\label{orderDef}
p=\log_{2}\frac{\left \| \mathbf{u}_{n}-\mathbf{u}_{n/2}\right\|}
{\left \| \mathbf{u}_{n/2} - \mathbf{u}_{n/4}\right\|},
\end{equation*}
with respect to~$\delta t$
and the timings for the whole computation. Here $\mathbf{u}_n$ is the 
computed solution at the final step in time.

Using our approach, it becomes possible to calculate $u^{(n)}(x,t)$
for large values of final time~$T$ due to the \emph{low-rank} approximation
of matrices~$\mathbf{\Phi}^{(k)}$ composed from
the columns of the integrand values (see Section~\ref{AlgorithmSect}).
That significantly reduces the computational cost.
For an example, for the last row of Table~\ref{Table1Cauchy}
iterations start from the calculation of the convolution on the range $[-16386, 16386)$ with
$32\,772\,000$ mesh points. This is reduced to the calculation
of $10$ (the rank) convolutions of two arrays with $8000$ elements.

As it can be seen from our results, the scheme has the second order in time.
It can be improved to higher orders by Richardson extrapolation 
on fine meshes~\cite{brezinski-extrapolation-1991,stoer-extrapolation-2002}.
Another way is to use other path integral formulations with high-order
 propagators~\cite{makri-pathsplit-1995, makri-quantdyn-2014}.

%
%Table 1
%
\begin{table}
\caption{Convergence rate for system~\eqref{CauchyEq}.
Accuracy of the cross approximation $\varepsilon = 10^{-10}$.
Direct convolutions start from $n=20$, $\sigma = 0.5$, range of 
final spatial domain is $[-2,2)$, $N_x = 4000$.
Dimension of the integral~\eqref{discr_uv} is labeled by~$n$,
$\delta t$ is a time step, $T$ is a final time for solution $u(x,T)$,
$\epsilon$~is an error estimated by the Richardson extrapolation, 
and $p$ is the order of the scheme for $\delta t$.
Ranks of the matrix $\mathbf{\Phi}^{(k)}$ from~\eqref{FBVcross} are
presented in column labeled by~$r$.
The CPU time for computation of the integral~\eqref{discr_uv} in \emph{all points} of
the mesh is reported in the last column.
\label{Table1Cauchy}
}
\centering
\begin{tabular}{crclccc}
\hline
\hline
$T$ & $n$ & $\delta t$  & $p$ & $\epsilon$ & $r$ & \textit{CPU Time} (min.)  \\
\hline
$1.0$ &     $32$ & $3.1 \cdot 10^{-2}$ & $-$         & $2.8 \cdot 10^{-4}$ & $10$ & $0.1$ \\
          &     $64$ & $1.6 \cdot 10^{-2}$ & $-$         & $7.0 \cdot 10^{-5}$ & $10$ & $0.2$ \\
          &   $128$ & $7.8 \cdot 10^{-3}$ & $1.997$ & $1.8 \cdot 10^{-5}$ & $10$ & $0.4$\\
          &   $256$ & $3.9 \cdot 10^{-3}$ & $1.999$ & $4.4 \cdot 10^{-6}$ & $10$ & $0.9$ \\
          &   $512$ & $2.0 \cdot 10^{-3}$ & $2.0$     & $1.1 \cdot 10^{-6}$ & $10$ & $1.8$ \\
          & $1024$ & $9.8 \cdot 10^{-4}$ & $2.0$     & $2.8 \cdot 10^{-7}$ & $10$ & $3.8$ \\
\hline
$20.0$&    $32$ & $6.3 \cdot 10^{-1}$ & $-$         & $4.1 \cdot 10^{-1}$ & $10$ & $0.1$ \\ 
          &     $64$ & $3.1 \cdot 10^{-1}$ & $-$         & $1.6 \cdot 10^{-1}$ & $10$ & $0.2$ \\
          &   $128$ & $1.6 \cdot 10^{-1}$ & $1.10$   & $4.8 \cdot 10^{-2}$ & $10$ & $0.4$ \\
          &   $256$ & $7.8 \cdot 10^{-2}$ & $1.68$   & $1.2 \cdot 10^{-2}$ & $10$ & $0.9$ \\
          &   $512$ & $3.9 \cdot 10^{-2}$ & $1.93$   & $3.1 \cdot 10^{-3}$ & $10$ & $1.9$ \\
          & $1024$ & $2.0 \cdot 10^{-2}$ & $1.98$   & $7.9 \cdot 10^{-4}$ & $10$ & $4.0$ \\
          & $2048$ & $9.8 \cdot 10^{-3}$ & $1.995$ & $2.0 \cdot 10^{-4}$ & $10$ & $8.0$ \\
          & $4096$ & $4.9 \cdot 10^{-3}$ & $1.999$ & $4.9 \cdot 10^{-5}$ & $10$ & $16.8$ \\
          & $8192$ & $2.4 \cdot 10^{-3}$ & $2.0$     & $1.2 \cdot 10^{-5}$ & $10$ & $37.5$ \\
\hline
\hline
\end{tabular}
\end{table}

%
%Figure 4
\begin{figure}
\includegraphics[width=10cm]{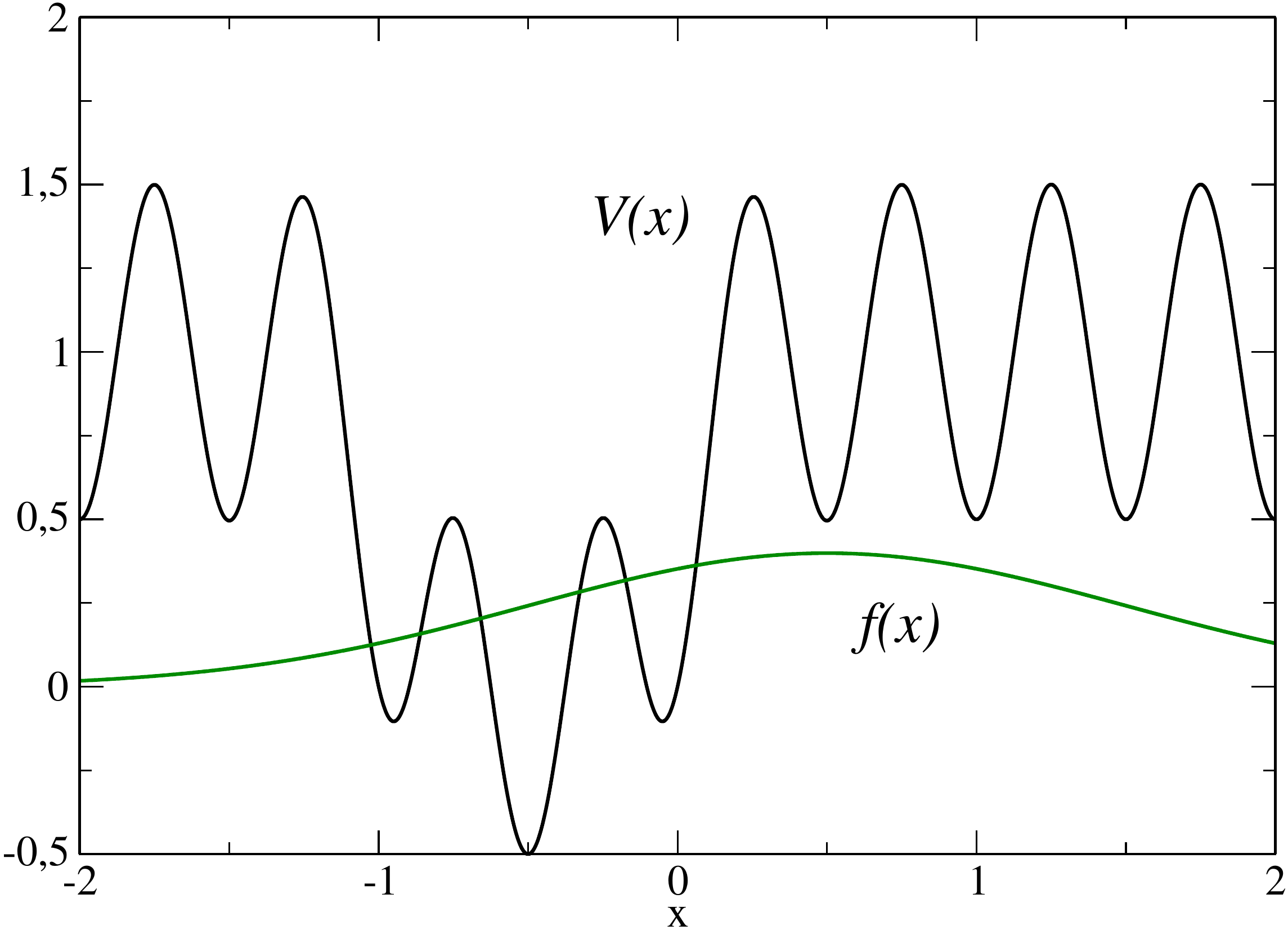}
\caption{
\label{Vfpict}
Potential $V(x)$ and initial distribution~$f(x)$ 
for periodic system with impurity~\eqref{Vf_imp}.
Potential oscillates on a free space. Functions $V(x)$ and $f(x)$
are relatively shifted to break the symmetry.
}
\end{figure}

%Figure 5
\begin{figure}
\includegraphics[width=10cm]{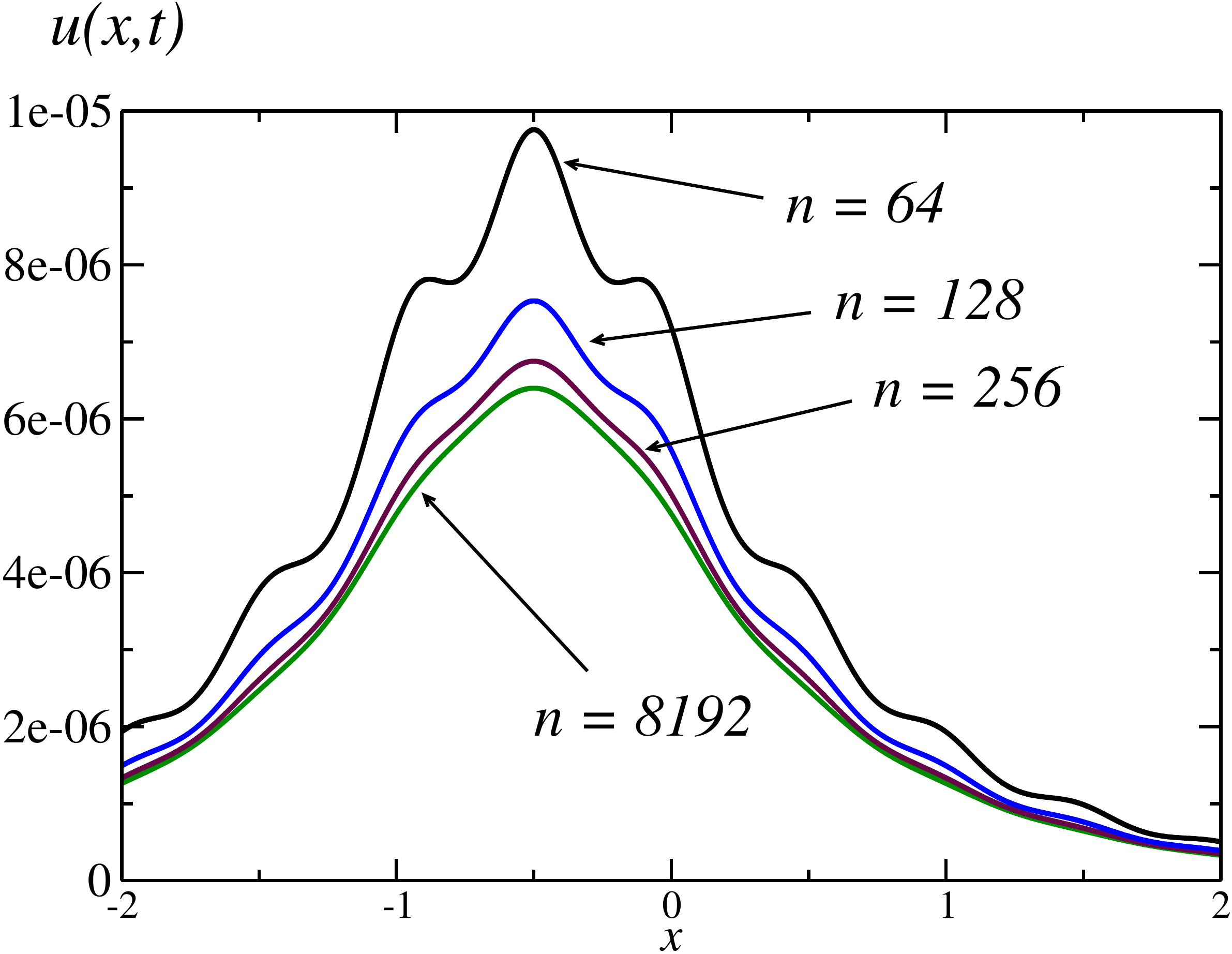}
\caption{
\label{Fig2Ux}
Convergence of solution $u(x,t)$ for nonperiodic potential with impurity~\eqref{Vf_imp}
for different~$n$. This results correspond to the data presented in Table~\ref{Table2Periodic}.
The number of spacial mesh points 
$M=2N_x =8000$ in the final range~$[-2,2)$.
The dissipation rate~\eqref{Vf_imp} leads to a decrease in
the norm of the distribution density. As seen in the picture,
the solution is far from the correct one for 
the dimensions $n = 64,128, 256$.
}
\end{figure}

%Figure Singular Values time
\begin{figure}
\includegraphics[width=10cm]{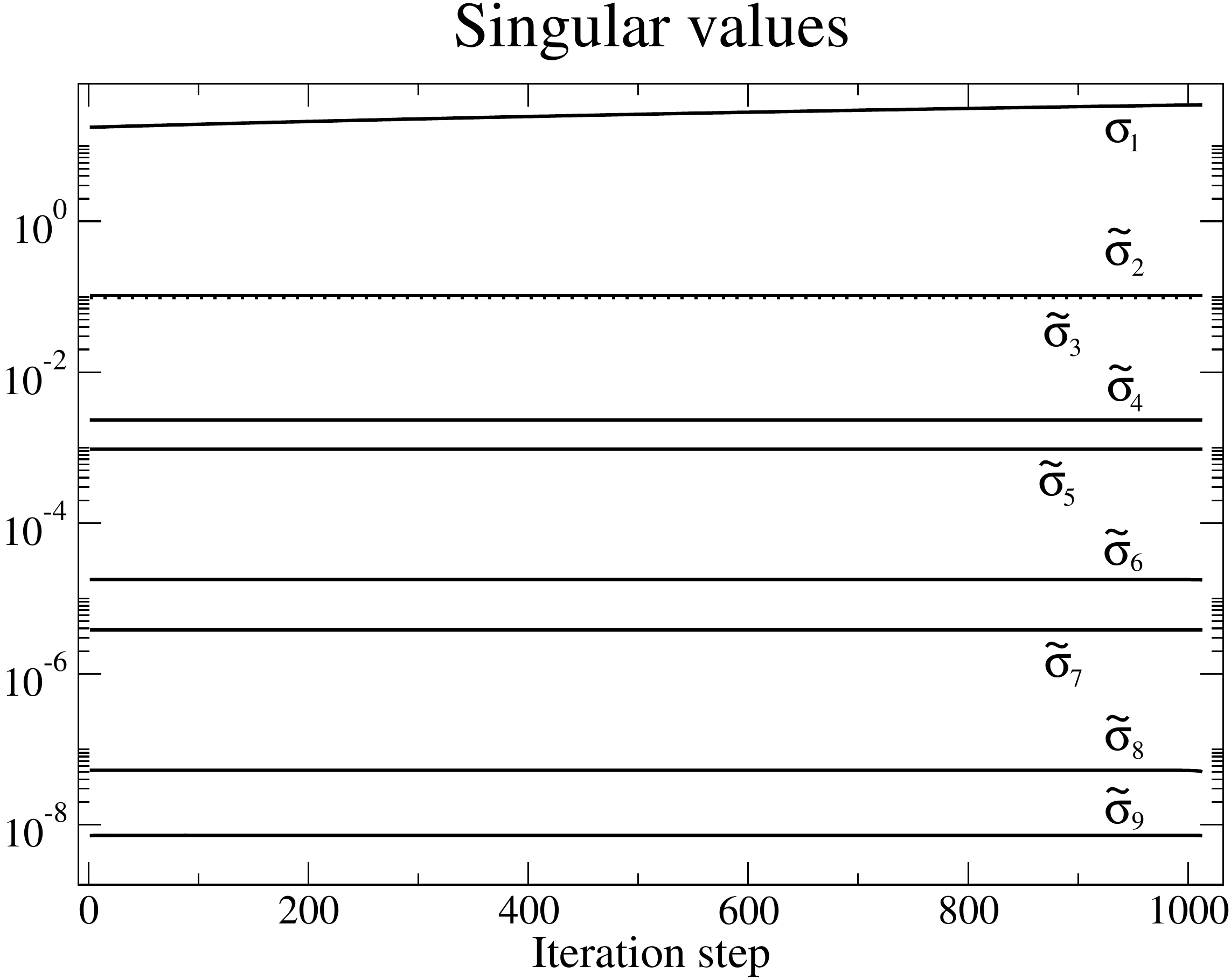}
\caption{
\label{Singulars}
The first few singular values (s.v.) of the matrix~\eqref{FBVcross} for system~\eqref{CauchyEq}
at each iteration step. 
The first s.v.~$\sigma_1$ is presented in the absolute value.
The other ones are given in the relative values as~$\sigma_i/\sigma_1$.
The values below the cross accuracy~$\varepsilon=10^{-10}$ are thrown out.
As it can be seen, approximate SVD-rank is similar to the cross 
rank (in the sense of criterion~\eqref{rankCrossCriterion}).
}
\end{figure}
% Figure

\subsection{Nonperiodic potential with impurity}
\label{ImpurPeriodicExmpl}
The dissipation rate $V(x,t)$ causes the creation and annihilation
of diffusing particles, as it follows from the main equation~\eqref{UxtDiffEqFull}. 
Without the Laplacian, which is responsible for the free diffusion, we have
\begin{equation*}
 \frac{\partial}{\partial t} u(x,t) = - V(x,t) u(x,t).
\end{equation*}
It can be seen, that the density of particles increases 
over time for $V(x,t) < 0$ and 
decreases for $V(x,t) > 0$ correspondingly.
The case $V(x,t) < 0$ may lead to 
an instability in the solution, because the integral
\begin{equation}
\label{IntVconvg}
\int_{-\infty}^{\infty} f(x+\xi) e^{-w_iV(x+\xi,\tau_{n-i})\delta t }e^{-\lambda \xi^2}d\xi,
\end{equation}
may diverge (see Eq.~\eqref{discr_uv}).
Therefore, when choosing~$V(x,t)<0$, one should make sure
 that the integral in~\eqref{IntVconvg} converges.

Consider the following problem (see Figure~\ref{Vfpict})
\begin{equation}
\label{Vf_imp}
V_{i}(x)=a + \sin^2 \left( \pi\left( \frac{x}{a} + 1\right) \right)
 -\frac{1}{1 + \left(\frac{x}{a} + 1 \right)^8},
\qquad
f_{i}(x) = \sqrt{ \frac{\beta}{\pi}} e^{-\beta (x-a)^2},
\qquad a = 0.5, \quad \beta = 0.5.
\end{equation}
It can be interpreted as a nonperiodic system with an impurity.
The term $V(x)$ does not decay in the spatial domain
and it is not periodic. Therefore the reduction of this problem to a bounded domain is not a trivial 
task and would require sophisticated artificial boundary conditions.

In~Table~\ref{Table2Periodic} we present results of numerical calculations,
which show the order of the numerical scheme.
In~Figure~\ref{Fig2Ux} we also present the computed solutions for different values of $n$.
Even in this case, the solution converges with the order $p=2$.
We also used the Richardson extrapolation of $u(x, T)$ for different $n$ 
to get higher order schemes in time.

%
%Table 2
%
\begin{table}
\caption{Convergence rate for system~\eqref{Vf_imp}.
Accuracy of the cross approximation $\varepsilon = 10^{-12}$.
Direct convolutions start from $n=20$, $\sigma = 0.25$, final domain is $[-2,2)$, $N_x = 8000$.
Dimension of the integral~\eqref{discr_uv} is labeled by~$n$,
$\delta t$ is the time step, $T=20$ is the final time.
The order of the scheme $p_2$ for $\delta t$ and the relative error
$\epsilon_2$~\eqref{RelEpsDef} are estimated from the original data computed 
by the algorithm from Section~\ref{AlgorithmSect}.
The next values $p_4$ and $\epsilon_4$ are estimated by the Richardson extrapolation.
As it can be seen, the scheme has the fourth order in time after the extrapolation.
The ranks of the matrix $\mathbf{\Phi}^{(k)}$ from\eqref{FBVcross} are given in the column labeled by~$r$.
The CPU time for computation of the integral~\eqref{discr_uv} in \emph{all points} of
the mesh is reported in the last column.
\label{Table2Periodic}
}
\centering
\begin{tabular}{rclclccc}
\hline
\hline
$n$ & $\delta t$  & $p_2$ & $\epsilon_2$  & $p_4$ & $\epsilon_4$ & $r$ & \textit{CPU Time} (min.)  \\
\hline
    $64$ & $3.1 \cdot 10^{-1}$ & $-$         & $-$                          & $-$       & $-$ & $9$ & $0.2$ \\
  $128$ & $1.6 \cdot 10^{-1}$ & $-$         & $8.3 \cdot 10^{-2}$ & $-$       & $-$ & $9$ & $0.3$ \\
  $256$ & $7.8 \cdot 10^{-2}$ & $1.47$   & $3.3 \cdot 10^{-2}$ & $-$       & $2.8 \cdot 10^{-3}$ & $9$ & $0.8$ \\
  $512$ & $3.9 \cdot 10^{-2}$ & $1.62$   & $1.1 \cdot 10^{-2}$ & $2.00$ & $7.0 \cdot 10^{-4}$ & $9$ & $1.7$ \\
 $1024$ & $2.0 \cdot 10^{-2}$ & $1.84$   & $3.1 \cdot 10^{-3}$ & $3.04$ & $8.6 \cdot 10^{-5}$ & $9$ & $3.6$ \\
 $2048$ & $9.8 \cdot 10^{-3}$ & $1.95$   & $8.1 \cdot 10^{-4}$ & $3.66$ & $6.8 \cdot 10^{-6}$ & $9$ & $7.0$ \\
 $4096$ & $4.9 \cdot 10^{-3}$ & $1.988$ & $2.0 \cdot 10^{-4}$ & $3.85$ & $4.7 \cdot 10^{-7}$ & $9$ & $14.7$ \\
 $8192$ & $2.4 \cdot 10^{-3}$ & $1.997$ & $5.1 \cdot 10^{-5}$ & $3.98$ & $3.0 \cdot 10^{-8}$ & $9$ & $33.0$ \\
\hline
\hline
\end{tabular}
\end{table}

\subsection{Monte Carlo experiments}
\label{MCsect}
In this section we present results of Monte Carlo simulation.
To estimate the solution in \emph{a fixed point} $x_0$ the
following formula is used
\begin{equation}
\label{uxt_MC}
u^{(n)}_{{ MC}}(x_0,T)=\frac{1}{K}\sum_{k=1}^{K}
f\left(\xi_{\langle k \rangle}(n)\right)
\! \prod_{i=0}^{n}
e^{ -w_{i} V (\xi_{\langle k \rangle}(i),  \tau_{n-i}) \delta t },
\end{equation}
\begin{equation*}
\xi_{\langle k \rangle}(i)=\xi_{\langle k \rangle 1} + \ldots + \xi_{\langle k \rangle i},
\end{equation*}
where each component of the vector 
$\mathbf{\xi}_{\langle k \rangle}=(\xi_{\langle k \rangle 1}, \ldots,
\xi_{\langle k \rangle n} )^{T}$ is 
independently taken from the normal distribution $\mathcal{N}(0,2\sigma\delta t)$
at each trial step $k: \, 1\le k \le K$, where $K$ being the number of trials.

Results for the exactly solvable model~\eqref{CauchyEq} are presented in
Table~\ref{Table3MC}.
We compare accuracy and timings for Monte Carlo and low-rank calculations.
It should be emphasized that in the Monte Carlo approach only one point 
of $u(x_0, T)$ is calculated for a fixed $x_0$ in one simulation,
while our approach allows to compute the whole array $u(x_i,T)$
on the whole mesh simultaneously. This numerical experiments have been done 
on a single CPU core without parallelization of the Monte Carlo algorithm
just to estimate the speedup of the low-rank computation.
More advanced realization such as quasi Monte Carlo methods can be used.
As it can be seen, the low-rank algorithm presented in Section~\ref{AlgorithmSect}
is much faster.
%
% Table 3 
\begin{table}
\caption{
Timings for system~\eqref{CauchyEq}.
Accuracy of the cross approximation $\varepsilon = 10^{-10}$.
Direct convolutions start from $n=30$, $\sigma = 0.5$, range of 
final spatial domain is $[-2,2)$, $N_x = 4000$.
Dimension of the integral~\eqref{discr_uv} is labeled by~$n$,
$\delta t$ is a time step, $T$ is a final time for solution $u(x_0,T)$
 computed in a fixed point $x_0$. Here $x_0 = 0$, $T = 1$.
The relative error~$\epsilon=|\tilde u(x_0,T) - u(x_0, T) | / |u(x_0, T)|$ is 
computed in one point $x_0$.
Time for one point calculation is presented for Monte Carlo approach~\eqref{uxt_MC}
and is estimated for the whole mesh array consisting of $M=2N_x=8000$~points (the last column).
For the low-rank computation the total timings are presented as well.
Monte Carlo simulation has been done with $K=10^9$ samples.
The low-rank results are labeled by LR, while the Monte Carlo results are labeled by MC.
\label{Table3MC}
}
\centering
\begin{tabular}{rclccr}
\hline
\hline
 $n$ & $\delta t$ & $u(x_0, T)$ & $\epsilon$ & \textit{CPU Time} \textit{($1$ point)} & \textit{CPU Time (total)} \\
\hline
    $32$ & $3.1 \cdot 10^{-2}$ & $0.6369899_{MC}$ & $5.8 \cdot 10^{-4}$ & $40.2$ min & $5.3 \cdot 10^3$ hrs (est.)\\
             &                                & $0.6369792_{LR}$  & $5.6 \cdot 10^{-4}$ &                    & $6$  sec \phantom{(est.)} \\
    $64$ & $1.6 \cdot 10^{-2}$ & $0.6367165_{MC}$ & $1.5 \cdot 10^{-4}$ & $79.1$ min & $1.0 \cdot 10^4$ hrs (est.)\\
             &                                & $0.6367099_{LR}$  & $1.4 \cdot 10^{-4}$ &                    & $13$ sec \phantom{(est.)}\\
  $128$ & $7.8 \cdot 10^{-3}$ & $0.6366653_{MC}$ & $7.2 \cdot 10^{-5}$ & $171$ min  & $2.2 \cdot 10^4$ hrs (est.)\\
             &                                & $0.6366423_{LR}$  & $3.5 \cdot 10^{-5}$ &                    & $26$ sec \phantom{(est.)}\\
  $256$ & $3.9 \cdot 10^{-3}$ & $0.6366388_{MC}$ & $3.0 \cdot 10^{-5}$ & $355$ min  & $4.7 \cdot 10^4$ hrs (est.)\\
             &                                & $0.6366254_{LR}$  & $8.9 \cdot 10^{-6}$ &                    & $53$ sec \phantom{(est.)}\\
  $512$ & $2.0 \cdot 10^{-3}$ & $0.6366218_{MC}$ & $3.2 \cdot 10^{-6}$ & $705$ min  & $9.4 \cdot 10^4$ hrs (est.)\\
             &                                & $0.6366212_{LR}$  & $2.2 \cdot 10^{-6}$ &                    & $1.8$ min \phantom{(est.)} \\
 &                                & $0.6366198_{exact}$      &                                &                                &  \\
\hline
\hline
\end{tabular}
\end{table}

\section{Conclusion and future work}

The presented results show that the proposed method is an efficient approach for solving
diffusion equations in a free space without artificial boundary conditions (ABC).
Instead of standard solvers based on the ABC designed for certain 
cases~\cite{dubach-bc-1996, xz-bc-2004}, 
our method is more universal one and
is applicable to a wide class of potentials as a unified approach.
It needs a constant memory size, which depends only on the final mesh size~$M$
and the rank~$r$ of the matrix of solution from~\eqref{FBVcross}
at each iteration step. Its complexity, then, is similar to the classical time-stepping 
schemes for the solution of the reaction-diffusion
equations in a bounded domain. It also shows a favourable scaling. 

It is natural to extend the approach presented in the current work to higher dimensions.
Then, instead of one-dimensional convolutions we will have to work with $d$-dimensional 
convolutions, where $d$ is the dimension of the problem. 
The extended domain will be $[-na, na]^d$, where $n$ is the number of time steps
(equal to the dimension of the path integral). 
Thus, for higher dimensions the solution can be treated as a $(d+1)$-dimensional tensor 
of size $M \times n \times \ldots \times n$.
Instead of the matrix low-rank approximation, stable low-rank factorization based on 
the tensor train decomposition~\cite{osel-tt-2011} could
be used, with the final cost approximately equal to the cost of 
the computation the convolutions on the small domain.

Finally, the most intriguing part of the work to be done
 is to apply the similar techniques to the Schr\"odinger equation.
There, the convolution is no longer a convolution with a 
Gaussian function. Thus, the problem is much more difficult
and our approach requires modifications.
The presented method can also be applied to path integrals arising in other application areas, 
including the financial mathematics. The main requirement is that 
the integrand depends on the sum of variables multiplied by a separable function.

\section*{Acknowledgements}
This work was partially supported
by Russian Science Foundation grant  14-11-00659.

\appendix

\section{The cross approximation of matrices}
\label{CrossApp}
Let $\mathbf{A} \in \mathbb{R}^{n \times m}$ and 
$\hat I=\{i_1,i_2\ldots, i_r \}$, $\hat J= \{ j_1,j_2\ldots j_r \}$, 
be subsets of $I=\{1,\ldots, n\}$ and $J=\{1,\ldots, m\}$, respectively, $r \le \min(n,m)$.
By $\hat {\mathbf{A}}=\mathbf{A}(\hat I, \hat J)$ we denote a submatrix of $\mathbf{A}$ formed by
the entries of $\mathbf{A}$ at the intersections of rows $i \in \hat I$
and columns $j \in \hat J$. In this paper we use the following concept of
\emph{the skeleton decomposition}~\cite{tee-mosaic-1996,
gtz-psa-1997, gt-psa-1995, gtz-maxvol-1997}. 
For any matrix $\mathbf{A} \in \mathbb{R}^{n \times m}$
of rank~$r$ there exist its decomposition
\begin{equation}
\label{skdec}
\mathbf{A}=\mathbf{B} \hat{\mathbf{ A}}^{-1} \mathbf{C}^{T},
\end{equation}
where
$\mathbf{B}=\mathbf{A}(I,\hat J)$,  $\mathbf{C}^{T}=\mathbf{A}(\hat I, J)$, and
$\hat{\mathbf{A}}=\mathbf{A}(\hat I, \hat J) \in \mathbb{R}^{r \times r}$ is a certain
submatrix of $\mathbf{A}$, such that $det \hat{\mathbf{A}} \ne 0$.
For the numerical reasons it is more effictive to work with orthogonal matrices.
The decomposition~\eqref{skdec} can be rewritten by the factorization of
the matrices $\mathbf{B}=\mathbf{Q}_B\mathbf{R}_B$ 
and $\mathbf{C}^T=\mathbf{R}_\mathbf{C}^T \mathbf{Q}^T_C$ 
by the \emph{QR-decomposition}, and by
further factorization of the 
rank-$r$ square matrix $\mathbf{R}_B \hat{\mathbf{ A}}^{-1} \mathbf{R}_C^T= 
\mathbf{U}_A \mathbf{\Sigma}_A \mathbf{V}_{A}^T$ by 
the \emph{singular value decomposition} 
(SVD)~\cite{golub-matrix-2012, demmel-linalg-1997}.
Thus, we will use the dyadic representation of~\eqref{skdec}
\begin{equation}
\label{CrossDecomp}
\mathbf{ A} = \mathbf{XY}^T, \qquad
A_{ij} = \sum_{q = 1}^{r} X_{iq} Y_{jq}, \qquad
\mathbf{X}=\mathbf{Q}_B \mathbf{U}_A \mathbf{\Sigma}^{1/2}_{A},
\qquad
\mathbf{Y}^{T}= \mathbf{\Sigma}^{1/2}_A \mathbf{V}_A^{T}\mathbf{Q}^T_C.
\end{equation}
If the rank of the matrix $\mathbf{A}$ is greater then $r$, 
in practice instead of exact equation~\eqref{skdec}
we consider approximation in some norm.
To obtain the decomposition~\eqref{CrossDecomp} in this case
we use the \emph{cross approximation} algorithm~\cite{tee-cross-2000,bebe-2000} 
based on the concept of the maximum 
volume submatrix (\textit{maxvol}) introduced in~\cite{gt-maxvol-2001,gostz-maxvol-2008}.
We have implemented our version of the 
algorithm available at~\cite{lits-deposit-2014, litsarev-cpc-2014}. 
Example of the usage of our code can be found in~\cite{litsarev-nlaa-2015}.
The new version of the code for complex and real matrices
will be aviable soon at~\cite{lits-dzcross-2015}.

The rank in the cross approximation technique is determined adaptively.
The algorithm starts from the guess rank~$r_0$ and at each iteration step~$k$
the subspace of vectors of~$\mathbf{B}$ and~$\mathbf{C}^{T}$
is doubled (they are chosen by the \textit{maxvol} subroutine,
which returns a set of $2r_k$ row (column) indeces of a submatrix of (almost) maximum volume). 
The next value of the rank~$r_{k+1}$, $r_{k+1}\le 2r_k$ is chosen
from the singular values of the matrix~$\mathbf{\Sigma}_{A}$ of size $2r_k \times 2r_k$
according to the following  criterion
\begin{equation}
\label{rankCrossCriterion}
r_{k+1}=\min_{1\le s \le 2r_k} \left\{s \, | \, \zeta(s) <  \varepsilon_{c} \right\},
\qquad
\zeta(s)=
\sqrt{ \frac{  \sum_{i=s+1}^{2r_k}\sigma^2_{i}  }
{\sum_{i=1}^{2r_k}\sigma^2_{i} } },
\qquad \zeta(2r_k)\equiv 0.
\end{equation}
The algorithm stops when
$|| \mathbf{\Sigma}_A^{(k)} - \mathbf{\Sigma}_A^{(k+1)} ||_2 < \varepsilon_{c} \mathbf{\Sigma}_A^{(k+1)}$
for the relative accuracy~$\varepsilon_c$.

Approximation~\eqref{CrossDecomp} can be obtained by the SVD decomposition
of the whole matrix~$\mathbf{A}$ with~$\mathcal{O}((n^2+m^2)m)$ complexity,
which is prohibitively slow.
In contrast, the rank-$r$ cross approximation requires only $\mathcal{O}((n+m)r)$ 
evaluations of the elements and $\mathcal{O}((n+m)r^2)$ additional operations.
This becomes crucial in practice, when the matrix element $A_{ij}$ is 
a time-consuming function to be calculated in a point $(i,j)$ for a finite time
or the given matrix is very large.
Existence of such an approximation and convergence of the cross algorithm
are discussed in Section~\ref{AlgorithmSect} and \ref{LowRankStrSect}.

\section{Numerical investigation of the low-rank structure of the solution basis set}
\label{LowRankStrSect}

Suppose, a function~$t(x) \in L_2(\mathbb{R})$ can be expanded into a series
\begin{equation}
\label{tcfexpan}
t(x)=\sum_{l=0}^{\infty}c_l \phi_l(x),
\qquad
c_l=\int_{-\infty}^{\infty} \! t(x) \, \phi_l(x)dx,
\end{equation}
where
\begin{equation}
\label{HermBasis}
\phi_l(x)= \small{\left( \frac{1}{2^l l! \sqrt{\pi} }   \right)^{\frac{1}{2}} } \,e^{-x^2/2} H_l(x),
\qquad H_0(x)=1,\quad H_1(x)=2x,\quad H_2(x)=4x^2 - 2,
\end{equation}
and $H_l(x),$ are Hermite polynomials,
with fast decaying coefficients~$c_l$, such that
for a given accuracy~$\varepsilon_1$
\begin{equation*}
\exists \, l_0 : \chi(l_0) < \varepsilon_1 \, \chi(0),
\qquad
\chi(l')=\sum_{l=l'+1}^{\infty}{c^2_l}.
\end{equation*}
And the approximated function
\begin{equation*}
\tilde t(x) =\sum_{l=1}^{l_0} c_l \phi_l(x),
\qquad
|| t(x) - \tilde t(x) || < \varepsilon_1
\end{equation*}
is of a canonical $\varepsilon_1$-rank~$l_0$.
Then, the question is about the rank structure of the matrix
$\mathbf{\Phi}_l$
%from~\eqref{FBVcross}, 
constructed as a reshape
of a corresponding one-dimensional basis vector~$\phi_l(x_i)$ defined 
on the uniform mesh~\eqref{ximesh}
\begin{equation}
\label{HermMatr}
(\Phi_l)_{ij}=\phi_l(x_{i+j\cdot M}).
\end{equation}
%It does not depend on the index of the iteration step.
If the matrices~$\{ \mathbf{\Phi}_l \}_{l=1}^{l_0}$ are the low-rank ones
then the matrix of a target function~$\tilde t(x)$
\begin{equation}
\label{Texpand}
\mathbf{T}=\sum_{l=1}^{l_0} c_l \, \mathbf{\Phi}_l, \qquad
T_{ij}=t(x_{i+j\cdot M}), 
\end{equation}
is of low-rank as well, which does not 
exceeds the upper bound $l_0\cdot r_{\max}$, where
$r_{\max}= \max_{1\le l \le l_0}( rank(\mathbf{\Phi}_l))$,
but practically, it is of order~$r_{\max}$.
In the Table~\ref{TableHermSVD}
we present first several singular values of matrix $\mathbf{\Phi}_l$.
As it can be seen, each matrix has the low-rank structure.
It would be nice to prove this numerical fact  theoretically.

Finally, to estimate the rank of the approximation~\eqref{Texpand}
one needs only to compute the coefficients~$c_l$ in the expansion~\eqref{tcfexpan}
and investigate their behaviour.
This idea is similar to the QTT approach applied to the Laplace and its
inverse operators in~\cite{gavkhor-cay-2011}.

\begin{table}
\caption{
Singular values (s.v.) of the matrix~\eqref{HermMatr} composed
from the discretized basis~\eqref{HermBasis}. The order of 
polynomials is labeled by~$l$, $\sigma_1$ is the first
(absolute) singular value, then $\sigma_i/\sigma_1$ 
corresponds to the following relative singular values.
The size of the matrix is $8000 \times 1024$.
As it can be seen (numerical) ranks does not exceed the 
value of~8 and grow from small to bigger~$l$.
\label{TableHermSVD}
}
\centering
\begin{tabular}{rccccccccc}
\hline
\hline
 $l$ & $\sigma_1$  & $\sigma_2/\sigma_1$ & $\sigma_3/\sigma_1$ & $\sigma_4/\sigma_1$ & 
                                   $\sigma_5/\sigma_1$ & $\sigma_6/\sigma_1$ & $\sigma_7/\sigma_1$ &  
                                   $\sigma_8/\sigma_1$ & $\sigma_9/\sigma_1$ \\ 
\hline
$0$ & $32.2$ & $0.96$ & $9.0 \cdot 10^{-5}$ & $8.7 \cdot 10^{-5}$ & $8.9\cdot 10^{-16}$ & 
           $8.7\cdot 10^{-16}$ & $7.4 \cdot 10^{-16}$ & $6.4 \cdot 10^{-16}$ & $5.5 \cdot 10^{-16}$ \\

 $1$ & $35.5$ & $0.77$ & $0.6 \cdot 10^{-3}$ & $ 0.6 \cdot 10^{-3}$ &  $1.1\cdot 10^{-14}$ & 
           $1.1\cdot 10^{-14}$ & $3.7 \cdot 10^{-16}$ & $3.4\cdot 10^{-16}$ & $2.8\cdot 10^{-16}$ \\

 $2$ &  $40.3$ & $0.48$ & $2.2\cdot 10^{-3}$ &  $1.6\cdot 10^{-3}$ & $8.6\cdot 10^{-14}$ &
            $8.3 \cdot 10^{-14}$ & $5.6 \cdot 10^{-16}$ & $5.0\cdot 10^{-16}$ & $4.6 \cdot 10^{-16}$\\

 $3$ & $38.0$ & $0.62$ & $0.8 \cdot 10^{-3}$ & $0.7 \cdot 10^{-3}$ & $6.2\cdot 10^{-13}$ & 
           $5.1\cdot 10^{-13}$ & $5.0 \cdot 10^{-16}$ & $4.4\cdot 10^{-16}$ & $4.4\cdot 10^{-16}$ \\ 

 $4$ & $37.0$ & $0.68$ & $2.1 \cdot 10^{-2}$ & $1.8\cdot 10^{-2}$ & $3.9 \cdot 10^{-12}$ & 
           $3.8 \cdot 10^{-12}$ & $6.6 \cdot 10^{-16}$ & $4.6 \cdot 10^{-16}$ & $4.5 \cdot 10^{-16}$ \\
           
 $5$ & $35.5$ & $0.77$ & $5.0\cdot 10^{-2}$ & $4.1\cdot 10^{-2}$ & $2.1 \cdot 10^{-11}$ & 
           $1.5\cdot 10^{-11}$ & $5.0\cdot 10^{-16}$ & $4.0\cdot 10^{-16}$ & $4.0\cdot 10^{-16}$ \\        

 $6$ & $38.3$ & $0.59$ & $9.2\cdot 10^{-2}$ & $5.8\cdot 10^{-2}$ & $9.2\cdot 10^{-11}$ & 
           $8.8\cdot 10^{-11}$ & $5.3\cdot 10^{-16}$ & $4.8\cdot 10^{-16}$ & $4.6\cdot 10^{-16}$ \\
           
 $7$ & $33.9$ & $0.83$ & $0.18$ & $0.13$ & $4.8\cdot 10^{-10}$ & $3.7\cdot 10^{-10}$ &
           $3.8\cdot 10^{-16}$ & $3.8 \cdot 10^{-16}$ & $3.4\cdot 10^{-16}$ \\        
           
 $8$ & $36.2$ & $0.68$ & $0.25$ & $0.09$ & $1.9 \cdot 10^{-9}$ & $1.8 \cdot 10^{-9}$ & 
           $4.1\cdot 10^{-16}$  & $4.0 \cdot 10^{-16}$ & $3.5\cdot 10^{-16}$ \\
         
 $9$ & $38.2$ &  $0.47$ & $0.29$ & $0.26$ & $7.3\cdot 10^{-9}$ & $7.0\cdot 10^{-9}$ & 
           $4.7\cdot 10^{-16}$ & $3.8\cdot 10^{-16}$ & $3.6\cdot 10^{-16}$ \\ 

$10$ & $28.7$ & $0.85$ & $0.60$ & $0.58$ & $3.6\cdot 10^{-8}$ & $2.9\cdot 10^{-8}$ &
            $7.9\cdot 10^{-16}$ & $4.7\cdot 10^{-16}$ & $4.5\cdot 10^{-16}$ \\

$11$ & $31.1$ & $0.73$ & $0.65$ & $0.35$ & $1.1\cdot 10^{-7}$ & $9.4\cdot 10^{-8}$ & 
            $5.5\cdot 10^{-16}$ & $5.0\cdot 10^{-16}$ & $4.7\cdot 10^{-16}$ \\

$12$ & $33.1$ & $0.65$ & $0.58$ & $0.25$ & $3.7\cdot 10^{-7}$ & $3.1\cdot 10^{-7}$ & 
            $5.3\cdot 10^{-16}$ & $5.0\cdot 10^{-16}$ & $4.8\cdot 10^{-16}$ \\

$13$ & $30.1$ & $0.69$ & $0.68$ & $0.51$ & $1.3\cdot 10^{-6}$ & $1.2\cdot 10^{-6}$ & 
            $7.0\cdot 10^{-16}$ & $5.7\cdot 10^{-16}$ & $5.3\cdot 10^{-16}$ \\

$14$ & $26.2$ & $0.95$ & $0.74$ & $0.67$ & $3.9\cdot 10^{-6}$ & $3.2\cdot 10^{-6}$ & 
            $7.2\cdot 10^{-16}$ & $6.9\cdot 10^{-16}$ & $6.1\cdot 10^{-16}$ \\

$15$ & $30.2$ & $0.74$ & $0.71$ & $0.38$ & $8.0\cdot 10^{-6}$ & $8.0\cdot 10^{-6}$ & 
            $6.2\cdot 10^{-16}$ & $6.0\cdot 10^{-16}$ & $5.3\cdot 10^{-16}$ \\

$16$ & $30.8$ & $0.82$ & $0.63$ & $0.21$ & $2.0 \cdot 10^{-5}$ & $1.7\cdot 10^{-5}$ & 
            $5.3\cdot 10^{-16}$ & $4.9\cdot 10^{-16}$ & $4.6\cdot 10^{-16}$ \\

$17$ & $28.4$ & $0.92$ & $0.67$ & $0.43$ & $5.5\cdot 10^{-5}$ & $5.2\cdot 10^{-5}$ & 
            $5.1\cdot 10^{-16}$ & $4.6\cdot 10^{-16}$ & $4.6\cdot 10^{-16}$ \\

$18$ & $28.5$ & $0.88$ & $0.68$ & $0.48$ & $1.3\cdot 10^{-4}$  & $1.1\cdot 10^{-4}$ & 
            $5.9\cdot 10^{-16}$ & $5.5\cdot 10^{-16}$ & $5.4\cdot 10^{-16}$ \\
            
$19$ & $30.0$ & $0.86$ & $0.51$ & $0.47$ & $2.6\cdot 10^{-4}$ & $2.3\cdot 10^{-4}$ & 
            $5.0\cdot 10^{-16}$ & $4.8\cdot 10^{-16}$ & $4.5\cdot 10^{-16}$ \\
 
$20$ & $31.6$ & $0.80$ & $0.56$ & $0.24$ & $5.3\cdot 10^{-4}$ & $3.4\cdot 10^{-4}$ &
            $5.3\cdot 10^{-16}$ & $4.7\cdot 10^{-16}$ & $4.5\cdot 10^{-16}$ \\
           
$21$ & $28.9$ & $0.90$ & $0.72$ & $0.23$ & $1.12\cdot 10^{-3}$ & $1.1\cdot 10^{-3}$ & 
            $4.6\cdot 10^{-16}$ & $4.4\cdot 10^{-16}$ & $4.3\cdot 10^{-16}$ \\

$22$ & $27.0$ & $0.98$ & $0.76$ & $0.47$ & $2.5\cdot 10^{-3}$ & $2.4\cdot 10^{-3}$ & 
            $6.2\cdot 10^{-16}$ & $5.6\cdot 10^{-16}$ & $4.8\cdot 10^{-16}$  \\
           
$23$ & $30.3$ & $0.67$ & $0.62$ & $0.59$ & $4.03\cdot 10^{-3}$ & $4.0\cdot 10^{-3}$ &
            $6.0\cdot 10^{-16}$ & $4.8\cdot 10^{-16}$ & $4.7\cdot 10^{-16}$ \\
           
$24$ & $28.0$ & $0.82$ & $0.78$ & $0.53$ & $8.3\cdot 10^{-3}$ & $6.9 \cdot 10^{-3}$ & 
            $8.2\cdot 10^{-16}$ & $7.6\cdot 10^{-16}$ & $4.8\cdot 10^{-16}$ \\
           
$25$ & $30.0$ & $0.77$ & $0.76$ & $0.23$ & $1.4\cdot 10^{-2}$ & $1.3\cdot 10^{-2}$ &
            $2.4\cdot 10^{-15}$ & $1.7\cdot 10^{-15}$ & $5.8\cdot 10^{-16}$ \\

$26$ & $29.4$ & $0.91$ & $0.67$ & $0.16$ & $2.4\cdot 10^{-2}$ & $2.0\cdot 10^{-2}$ & 
            $7.7\cdot 10^{-15}$ & $6.8\cdot 10^{-15}$ & $6.3\cdot 10^{-16}$ \\

$27$ & $27.4$ & $0.88$ & $0.87$ & $0.34$ & $4.2\cdot 10^{-2}$ & $3.7\cdot 10^{-2}$ & 
            $2.6\cdot 10^{-14}$ & $2.4\cdot 10^{-14}$ & $4.5\cdot 10^{-16}$ \\
            
$28$ & $29.4$ & $0.75$ & $0.71$ & $0.50$ & $5.8\cdot 10^{-2}$ & $5.7\cdot 10^{-2}$ &
            $6.8\cdot 10^{-14}$ & $5.6\cdot 10^{-14}$ & $5.6\cdot 10^{-16}$ \\
         
$29$ & $28.4$ & $0.81$ & $0.71$ & $0.56$ & $8.9\cdot 10^{-2}$ & $8.4\cdot 10^{-2}$ &
            $2.2\cdot 10^{-13}$ & $1.5\cdot 10^{-13}$ & $5.5 \cdot 10^{-16}$ \\
         
$30$ & $27.8$ & $0.86$ & $0.79$ & $0.44$ & $0.14$ & $8.8\cdot 10^{-2}$ & 
            $6.2\cdot 10^{-13}$ & $5.6\cdot 10^{-13}$ & $5.1\cdot 10^{-16}$ \\
                           
$31$ & $28.0$ & $0.96$ & $0.74$ & $0.20$ & $0.17$ & $0.12$ & 
            $1.2 \cdot 10^{-12}$ & $1.2\cdot 10^{-12}$ & $6.1\cdot 10^{-16}$ \\
     
$32$ & $26.3$ & $0.99$ & $0.89$ & $0.27$ & $0.19$ & $0.14$ & 
            $5.4\cdot 10^{-12}$ & $5.3\cdot 10^{-12}$ & $7.8\cdot 10^{-16}$ \\
\hline
\hline
\end{tabular}
\end{table}

\section{Proof of the lemmas}
\label{App:proveLem}
% Proof Lemma 1
\begin{proof}{(Lemma~\ref{lem2}).}
For the basis set $\{ \mathbf{u}_{i} \}_{i=0}^{r_1-1}$ the following equality holds
\begin{equation}
\label{lmuiDecompose}
\mathbf{l}_{m}= \sum_{i=0}^{r_1-1} \alpha_{mi} \mathbf{u}_i.
\end{equation}
Then, according to~\eqref{Aalpfab}
\begin{equation}
\left[ \mathbf{l}^T_{m}, \mathbf{0} \right]_H = 
\left[ \sum_{i=0}^{r_1-1} \alpha_{mi} \mathbf{u}^T_i, \mathbf{0} \right]_H = 
 \sum_{i=0}^{r_1-1} \alpha_{mi} 
 \left[\mathbf{u}^T_i, \mathbf{0} \right]_H, \qquad
 \Leftrightarrow \qquad
 \mathbf{L}_m=\sum_{i=0}^{r_1-1} \alpha_{mi} 
 \mathbf{U}_i, \qquad \forall m \in [0, k-1].
\end{equation}
\end{proof}
% End  Proof Lemma 1
% Proof Lemma 2
\begin{proof}{(Lemma~\ref{lem3}).}
From the equality
$\mathbf{r}_{m}= \sum_{i=0}^{r_2-1} \beta_{mi} \mathbf{w}_i$
it follows that
\begin{equation*}
\left[\mathbf{0}^T, \mathbf{r}_{m} \right]_H = 
\left[\mathbf{0}^T, \sum_{i=0}^{r_2-1} \beta_{mi} \mathbf{w}_i \right]_H = 
 \sum_{i=0}^{r_2-1} \beta_{mi} 
 \left[ \mathbf{0}^T, \mathbf{w}_i \right]_H, \qquad
 \Leftrightarrow \qquad
 \mathbf{R}_m=\sum_{i=0}^{r_2-1} \beta_{mi} 
 \mathbf{W}_i, \qquad \forall m \in [0, k-1].
\end{equation*}
\end{proof}
% End  Proof Lemma 2
% Proof Lemma 3
\begin{proof}{(Lemma~\ref{lem4}).}
From definition~\eqref{lrDef} and the decomposition
$\mathbf{l}_{m}=\sum_{i=0}^{r_1-1}\gamma_{mi}\mathbf{u}_i, 
\phantom{x} \forall j_m \in [0,k]$,
it follows that
\begin{equation*}
\begin{pmatrix}
\mathbf{r}_{m} \\
l_{m,(M-1)}
\end{pmatrix}
=\mathbf{l}_{m}=
\sum_{i=0}^{r_1-1}\gamma_{mi}\mathbf{u}_i=
\sum_{i=0}^{r_1-1}\gamma_{mi}
\begin{pmatrix}
\mathbf{w}_i \\
u_{m,(M-1)}
\end{pmatrix},
\qquad
\Rightarrow
\qquad
\mathbf{r}_{m}=\sum_{i=0}^{r_1-1}\gamma_{mi} \mathbf{w}_i.
\end{equation*}
\end{proof}

\section{Solution for the harmonic oscillator potential}
\label{App:Osc}
In this section we analytically integrate 
equations~\eqref{Integral_Iter_k}, \eqref{Phi_iter_k}, \eqref{Init_F_n}
with the initial condition~\eqref{OscfV}. 

Let us define $F^{(n)}_k(x)\equiv \Psi^{(n)}_{k}(x)$ for harmonic potential~\eqref{OscfV}.
Starting from $k=n$,
\begin{equation*}
\Psi^{(n)}_{k}(x)= 
{\displaystyle \sqrt{\frac{\lambda}{\pi}}
\sqrt{\frac{\beta}{\pi}} }
\int_{-\infty}^{\infty}
 e^{-\beta(x+\xi)^2} e^{ -w_n (x+\xi)^2 \delta t  }
\, e^{-\lambda \xi^2 } d\xi,
\end{equation*}
and making use of the integral
\begin{equation*}
P(\alpha, \beta, y)=
\int_{-\infty}^{\infty}
p(\beta,y+\xi) p(\alpha, \xi) d\xi
=p\left({\textstyle \frac{\alpha \beta}{\alpha + \beta}} ,y\right)=
{\textstyle \sqrt{\frac{\alpha \beta}{\pi (\alpha + \beta)}} }
e^{-\frac{\alpha \beta}{\alpha + \beta} y^2},
\qquad y \in \mathbb{R}, \qquad \alpha, \beta > 0,
\end{equation*}
where $p(\alpha,x)$ is defined in~\eqref{Iksum}, we obtain
\begin{equation*}
\Psi^{(n)}_{n}(x)= \sqrt{\frac{\beta_{n}}{\gamma_n }}
\sqrt{\frac{\beta_{n-1}}{\pi }}  e^{-\beta_{n-1} x^2},
\qquad
\beta_{n-1}=\frac{\lambda \gamma_n}{\lambda + \gamma_n},
\qquad
\gamma_n = \beta_{n} + w_n \delta t,
\qquad
\beta_{n}=\beta.
\end{equation*}
For the next $k=n-1$, we have
\begin{equation*}
\Psi^{(n)}_{n-1}(x)= 
\sqrt{\frac{\lambda}{\pi}} \sqrt{\frac{\beta_{n}}{\gamma_n }}
\sqrt{\frac{\beta_{n-1}}{\pi}}
\int_{-\infty}^{\infty}
 e^{-\beta_{n-1}(x+\xi)^2} e^{ -w_{n-1} \frac{(x+\xi)^2}{1+\delta t} \delta t  }
\, e^{-\lambda \xi^2 } d\xi,
\end{equation*}
\begin{equation*}
\Psi^{(n)}_{n-1}(x)= 
\sqrt{\frac{\beta_{n}}{\gamma_n }}
\sqrt{\frac{\beta_{n-1}}{\gamma_{n-1} }}
\sqrt{\frac{\beta_{n-2}}{\pi }}  e^{-\beta_{n-2} x^2},
\qquad
\beta_{n-2}=\frac{\lambda \gamma_{n-1}}{\lambda + \gamma_{n-1}},
\qquad
\gamma_{n-1} = \beta_{n-1} + w_{n-1} \frac{\delta t}{1+\delta t}.
\end{equation*}
By induction, we conclude that
\begin{equation*}
\label{PsiSlnk}
\Psi^{(n)}_{k}(x)= 
\Gamma^{(n)}_{k}
\sqrt{\frac{\beta_{k-1}}{\pi }}  e^{-\beta_{k-1} x^2},
\qquad
\beta_{k-1}=\frac{\lambda \gamma_{k}}{\lambda + \gamma_{k}},
\qquad
\gamma_{k} = \beta_{k} + w_{k} \frac{\delta t}{1+(n-k)\delta t},
\end{equation*}
\begin{equation*}
\Gamma^{(n)}_{k}=\sqrt{\frac{\beta_{n}}{\gamma_n }}
\sqrt{\frac{\beta_{n-1}}{\gamma_{n-1} }}
\cdot \ldots \cdot
\sqrt{\frac{\beta_{k}}{\gamma_{k} }},
\qquad
1 \le k \le n
\end{equation*}

%\section*{References}
\bibliographystyle{model6-num-names}
\bibliography{misha,pathrefs,our,tensor}

\end{document}